\documentclass[12pt, a4paper]{article}

\usepackage[utf8]{inputenc}
\usepackage{amsmath}
\usepackage{float}
\usepackage[english]{babel}
\usepackage{graphicx}
\usepackage{multirow}
\usepackage{upgreek}
\usepackage{tikz}
\usepackage{ mathrsfs }
\usepackage{mathtools}
\usepackage{caption}
\usepackage{subcaption}
\usepackage{pifont}
\usepackage{amssymb}
\usepackage{enumerate}
\usepackage{authblk}
\usepackage{hyperref}
\usepackage[square,numbers]{natbib}
\usepackage{cleveref}
\usepackage{ntheorem}
\newtheorem{theorem}{Theorem}

\newtheorem{lemma}{Lemma}
\newtheorem{definition}{Definition}
\newtheorem*{proof}{Proof:}
\newtheorem{remark}{Remark}
\newtheorem{hypothesis}{Hypothesis}
\newtheorem{proposition}{Proposition}

\usepackage[a4paper, margin=1in]{geometry}
\title{ Existence and Stability Theory of a Neurologically-Inspired Parabolic PDE Model with a Nonlinear Time-Delayed Boundary Condition}
\date{}

\author[1, a]{Gangadhara Boregowda}
\author[1, b]{Michael R. Lindstrom}
\affil[1]{School of Mathematical and Statistical Sciences,  The University of Texas Rio Grande Valley, 
Edinburg, TX, USA.}
\affil[a]{
\texttt{gangadhara.boregowda@utrgv.edu}}
\affil[b]{
\texttt{mike.lindstrom@utrgv.edu}}

\begin{document}
\maketitle
In this paper, we establish the existence of a positive, bounded solution for a class of parabolic partial differential equations with nonlinear boundary conditions, where the boundary conditions depend on the solution on the boundary at a time $\tau \geq 0$ in the past. These equations model the production dynamics of a protein species by a single cell, where a feedback mechanism downregulates the protein's production. Furthermore, we analyze the stability of a non-trivial steady-state solution and provide sufficient conditions on the nonlinearity parameter, boundary flux, and time delay that ensure the occurrence of a Hopf bifurcation.

\textbf{keywords:} Time delay,  Hopf bifurcation, Nonlinear PDE, Stability.

\section{Introduction}
Mathematical modeling of biological systems is an important area of research. The systems themselves are often very complex, obscuring the individual effects of isolated mechanisms, and requiring costly and lengthy experiments. By carefully considering relevant mechanisms, a mathematical model can be built to combine those that are most relevant in order to gain insight into the larger system.

Partial Differential Equations (PDEs) are fundamental tools in mathematical modeling, particularly for systems influenced by both spatial and temporal variables. They find applications across various disciplines, including bioheat transfer for modeling thermal processes in biological tissues \cite{Pennes1948}, population dynamics \cite{holmes1994}, the Bhatnager-Gross-Krook (BGK) model in the kinetic theory of gases \cite{tiwari2020}, and, notably, the Navier-Stokes equations for fluid flows \cite{bertozzi2001}. A major challenge in PDE-based modeling lies in establishing the existence, uniqueness, and stability of solutions. In practice, especially in engineering \cite{chung2002}, PDE-based models are often analyzed numerically under the assumption that a solution exists within a defined solution space.

Nonlinear time delayed PDEs play a significant role in modeling various real-world phenomena, including compartmental systems \cite{mahaffy1984,pao1987, wu2012}, size-structured models, and species population models \cite{magal2018}, among others. The study of such equations is crucial for understanding the underlying systems and developing numerical methods to approximate their physical solutions. Recently,  reaction diffusion equations with nonlinear boundary conditions have garnered increasing attention, with various aspects of this topic being extensively studied \cite{guo2021, guo2023, arrieta1999}. The existence of solutions to nonlinear delayed PDEs has been explored in mathematical literature \cite{pao2002, magal2018, travis1974}. For instance, the framework of abstract parabolic equations in Banach spaces has proven effective in addressing this problem \cite{magal2018, brezis2011}. Notably, C.V. Pao advanced the study of nonlinear parabolic and elliptic PDEs by employing the upper and lower solution method, as detailed in \cite{pao2012, pao2002, pao2007}. This method offers a robust and systematic approach to proving the existence of classical solutions for nonlinear parabolic PDEs. 

An essential aspect of time-dependent systems is understanding how the solution changes over time. The parameters of these systems play a critical role in shaping the solution's behavior throughout its evolution, highlighting the importance of identifying the conditions that govern their long-term dynamics. In recent decades, the interaction between diffusion and time delay has been thoroughly studied \cite{chen2012, chen2018, guo2015, guo2021oscillatory}, with delay-induced Hopf bifurcations serving as a key tool for understanding the dynamics of real-world systems. Patrick Guidotti and Sandro Merino \cite{guidotti1997} investigated  Hopf bifurcations in reaction-diffusion equations with nonlinear boundary conditions by analyzing the spectrum of the corresponding eigenvalue problem. Similarly, Hui et al. \cite{hui2022} extended this analysis to delayed reaction-diffusion equations with homogeneous Neumann boundary conditions, focusing on the stability of steady state solutions and the existence of Hopf bifurcations by examining the principal eigenvalue of an associated elliptic operator. Chaochao Li and Shangjiang Guo \cite{li2024} explored the bifurcation and stability of reaction-diffusion-advection models, identifying specific conditions on time delay under which Hopf bifurcations occur.

Building on the existing literature, this paper investigates a novel class of PDE characterized by nonlinear, time-delayed boundary conditions. The focus of this study is restricted to examining the existence and Hopf bifurcation analysis of solutions to the proposed model. The formulation of the model in a half-space domain, combined with the assumption that the initial concentration is a bounded continuous function, naturally motivates the use of Green's functions to establish the existence of solutions. 

The structure of the paper is as follows: Section 2 presents the background of the model and outlines the contributions of the study. Section 3 introduces the preliminaries and key definitions. The main results, including existence, uniqueness, Hopf bifurcations, and numerical simulations, are detailed in Section 4. Finally, the conclusions and potential directions for future work are discussed in Section 5. Additional supporting results are provided in the Appendix.
\section{Model equation and background}
In this paper, we concern ourselves with a novel type of delayed partial differential equation, whereby the prescribed flux at a boundary is a nonlinear function of the solution at the boundary in the past. An example of this sort of model was previously studied in our group \cite{miller2024} in the context of a prion disease. In prion diseases, the misfolded scrapie prion protein (\(\textnormal{PrP}^\text{Sc}\)) accumulates in the extracellular space and interacts with the normal cellular prion protein (\(\textnormal{PrP}^\text{C}\)) produced by neurons, converting \(\textnormal{PrP}^\text{C}\) into \(\textnormal{PrP}^\text{Sc}\). The excessive accumulation of this toxic protein triggers the Unfolded Protein Response (UPR), a cellular stress mechanism that slows or halts the synthesis of the normal prion protein. This reduces the rate \(\textnormal{PrP}^\text{Sc}\) is created, allowing it to be cleared over time. Once the toxic protein is cleared, the production of the healthy prion protein resumes. This feedback process can induce oscillations in protein concentrations \cite{miller2024}.
For this paper, we create a simpler Toy Model, to study the basic structure of such in the  system. We set our focus on a single cell that produces a particular protein species. This species itself downregulates its own production with a delay. More precisely, we consider a one-dimensional model with the cell membrane at position $x=0$ and interstitial fluid occupying $x>0$. Then, for protein concentration $P(x,t)$ at positions $x\geq 0$ and times $t\geq 0$, we consider the following system:  
\begin{align}
    P_t&=D P_{xx} - k P, \label{eq:P} \\
    P(\infty,t)&=0, \label{eq:Pinfty} \\
    D \frac{\partial P}{\partial x}\Big|_{x=0}&=\frac{-A}{1+(P(0,t-T_0)/P_0)^m}, \label{eq:Pflux} \\
    P(x,0)&=F(x), \label{eq:Pic} \\
    P(0,t)&=G(t), \quad t \in [-T_0,0]. \label{eq:Phist}
\end{align}
In this model Eq. \eqref{eq:P} models the protein concentration as a reaction diffusion equation with diffusivity $D>0$ and clearance rate $k>0$; Eq. \eqref{eq:Pinfty} imposes that the concentration goes to zero in the far-field; Eq. \eqref{eq:Pflux} has the flux depending on the concentration at a time $T_0 \geq 0$ in the past, where $A>0$ is the maximum possible flux, $P_0>0$ is a protein-sensitivity switch threshold, and the nonlinearity parameter $m>0$ governs the speed of the switch; the initial concentration is given by Eq. \eqref{eq:Pic}; and the history function of the boundary is given by Eq. \eqref{eq:Phist}. Note that for $P<P_0$, resp., $P>P_0$, for large $m$, the protein flux into the domain is $\approx A$, resp., $\approx 0$. Through a standard nondimensionalization, we arrive at our system
\begin{align}
   \label{eq:Mod1a} q_t&=q_{xx} - q, \\
    \label{eq:Mod1b} q(\infty,t)&=0, \\
   \label{eq:Mod1c} q_x(0,t)&=\frac{-\alpha}{1+q(0,t-\tau)^m}, \\
    \label{eq:Mod1d} q(x,0)&=f(x), \\
   \label{eq:Mod1e}  q(0,t)&=h(t) \quad t \in [-\tau,0],
\end{align}
where $\tau=kT_0$ is the dimensionless delay parameter, $\alpha = \frac{A}{P_0 \sqrt{Dk}}$ is the dimensionless flux parameter; $q$, $f$, and $h$ are dimensionless renditions of $P$, $F$, and $G$; and $x$ and $t$ are now dimensionless.

We assume that 
\begin{hypothesis}\label{hy:gen} 
(General hypothesis)
 \begin{enumerate}
     \item  $\forall x \in [0, \infty)$,  $f(x)\geq0$ and $f(x) \in \mathcal{C}\left([0,\infty)\right) \bigcap \mathcal{L}_{\infty}(\left[0,\infty)\right)$;
     \item $\forall t \in [-\tau, 0]$, $h(t)\geq 0$ and $h(t) \in \mathcal{C}\left([-\tau,0]\right)$;
     \item $h(0)=f(0)$ (compatibility condition); and
     \item $\alpha >0$, $m>0$ and $\tau\geq 0$.
 \end{enumerate}
\end{hypothesis}
This paper investigates the existence and uniqueness of classical solutions to the model represented by Eqs. \eqref{eq:Mod1a}--\eqref{eq:Mod1e}, where the flux at the membrane depends on either the past value at the membrane ($\tau > 0$) or the present value at the membrane ($\tau = 0$). Since $m>0$, the flux boundary condition in Eqs. \eqref{eq:Mod1a}--\eqref{eq:Mod1e} is nonlinear in $q(0,t-\tau)$, which significantly increases the model's complexity. We establish the existence of a positive, bounded solution and the asymptotic behavior of the solution to Eqs. \eqref{eq:Mod1a}--\eqref{eq:Mod1e} by considering the cases $\tau>0$ and $\tau=0$. When $\tau>0$, Eqs. \eqref{eq:Mod1a}--\eqref{eq:Mod1e} become linear in the interval $[0, \tau]$ because the flux boundary is known. To prove the existence of the solution in this case, we employ Theorem 1.2 from Chapter 7 of \cite{pao2012} and the ladder argument. When $\tau=0$, Eqs. \eqref{eq:Mod1a}--\eqref{eq:Mod1e} remain nonlinear, and we apply the method of lower and upper solutions to demonstrate the existence and uniqueness of the solution.  This approach provides a powerful framework for analyzing complex nonlinear systems by constructing two bounding sequences that converge to the actual solution of the PDE. Under Hypothesis~\ref{hy:gen}, we perform a bifurcation analysis of the steady state solution of Eqs. \eqref{eq:Mod1a}--\eqref{eq:Mod1e}, using the time delay as the bifurcation parameter. Additionally, we derive conditions on $m$ and $\alpha$ that ensure the occurrence of a Hopf bifurcation induced by the time delay. The key results of this paper and their respective contributions are summarized in Table ~\ref{tb:contr}. The global constants referenced in this study are summarized in Table~\ref{tb:notation} for clarity and ease of reference.

\begin{table}[h]
\caption{Main results and contributions}
\label{tb:contr}
\centering
\setlength{\tabcolsep}{13pt}
\begin{tabular}{|p{3cm}|p{11cm}|}
\hline
\textbf{Result} & \textbf{Contribution} \\
\hline
Propositions \ref{pr:lusol1} and \ref{pr:lusol2} & Establish lower $(\check{q})$ and upper $(\hat{q})$ solutions for Eqs.~\eqref{eq:Mod1a}--\eqref{eq:Mod1e} when $\tau=0$, forming the foundation for subsequent analysis. \\
\hline
Lemma \ref{le:monot} & Proves the monotonicity of the iterated sequence, a critical property used to confirm the existence of solutions when $\tau=0$. \\
\hline
Theorem \ref{th:ExUnNde} & Provides a formal proof for the existence of solutions to Eqs.~\eqref{eq:Mod1a}--\eqref{eq:Mod1e} under the condition $\tau=0$. \\
\hline 
Theorem \ref{th:UETM} & Demonstrates the existence of solutions for Eqs.~\eqref{eq:Mod1a}--\eqref{eq:Mod1e} when $\tau>0$. \\
Theorem \ref{th:Glob} & Proves the uniqueness of the solution to Eqs.~\eqref{eq:Mod1a}--\eqref{eq:Mod1e}, ensuring that the problem is mathematically well-posed. \\ 
\hline
Theorem \ref{th:hobi} & Provides sufficient conditions for the occurrence of Hopf bifurcation, revealing insights into the stability and periodicity of solutions. \\
\hline
\end{tabular}
\end{table}

\begin{table}[h]
\caption{List of global notations and their descriptions}
\label{tb:notation}
\centering
\setlength{\tabcolsep}{13pt}
\begin{tabular}{|p{3cm}|p{11cm}|}
\hline
\textbf{Notation} & \textbf{Description} \\
\hline
$\alpha, m, \tau$ & System parameters: maximum flux, nonlinearity, and delay parameters \\
\hline
$M$ & $\sup \{f(x): x\in \Omega\}$ \\
\hline
$c$ & Positive real root of $c + c^{m+1} - \alpha = 0$, for given $\alpha, m > 0$ \\
\hline
$\Gamma$ & Lipschitz constant for the function $g(0,t,q): [0, T] \times \langle \check{q}, \hat{q} \rangle \rightarrow \mathbb{R}$ defined by $g(0,t,q) = \frac{-\alpha}{1 + [q(0,t)]^m}$ \\
\hline
$Q$ & $Q = \frac{\alpha m c^{m-1}}{(1 + c^{m})^2}$, which emerges during the linearization process and plays a crucial role in ensuring the stability of the solution \\
\hline
$T$ & Indicates maximum interval of existence of solution: when $T < \infty$, the interval is $[0, T]$; when $T = \infty$, it is $[0, \infty)$ \\
\hline
\end{tabular}
\end{table}

\section{Theoretical background}
In this section, we introduce the key concepts necessary for analyzing Eqs. \eqref{eq:Mod1a}--\eqref{eq:Mod1e}: the domain of interest, solution spaces, definitions of upper and lower solutions, stability, and properties of Green's functions. These foundational tools are crucial for the forthcoming analysis. Readers already familiar with these concepts may proceed directly to the next section  

\vspace{1cm}
To describe the system we consider $\Omega:=[0, \infty)$ with the boundary (cell membrane) $\partial \Omega:=\{0\}$. For any time $T>0$, we set 
\begin{equation*}
    \Omega_T= \Omega \times (0, T) , \  \partial \Omega_T^{*}=  \left(\Omega \times \{0\}\right) \cup \left(\partial \Omega \times (0,T)\right), \ \text{and} \ H_T=\{0\} \times [-\tau, 0].
\end{equation*}
We use the following spaces in our analysis
\begin{itemize}
    \item Let $\mathcal{C}(\Omega_T)$ denote the set of all continuous functions on $\Omega_T$
    \item Let   $\mathcal{C}^{2,1}(\Omega_T)$  denote the set of all functions that are twice continuously differentiable in space and once continuously differentiable in time 
    \item The function sector for $f_1(x,t)$ and $f_2(x,t)$ is defined as  
    \begin{equation*}
        \langle f_1,f_2 \rangle=\{f \in \mathcal{C}^{2,1}\left(\Omega_{T}\right):f_1\leq f \leq f_2\} 
    \end{equation*}
\end{itemize}
\subsection{Upper and lower solutions}
To establish the existence of solutions for nonlinear partial differential equations, we employ the method of upper and lower solutions. Given the unbounded nature of the domain, we incorporate a growth condition in the formulation, as outlined in \cite{pao2012}. The exact definition can be applied in a bounded domain without imposing a growth condition \cite{pao2007}.

\begin{definition}
    A function $\hat{q} \in \mathcal{C}^{2,1}\left(\Omega_{T}\right)$ is called an upper solution of Eqs. \eqref{eq:Mod1a}--\eqref{eq:Mod1d} with $\tau=0$ if it satisfies the inequalities
   \begin{align}
\label{eq:uppera}\hat{q}_{t}-\hat{q}_{x x} +\hat{q}&\geq 0,  \\
\label{eq:upperb} \hat{q}( \infty,t) & \geq 0,\\
\label{eq:upperc}\hat{q}_{x}(0,t) & \geq \frac{-\alpha}{1+[\hat{q}(0, t)]^{m}}, \\
\label{eq:upperd}\hat{q}(x,0) & \geq f(x), 
\end{align}
and the growth condition 
\begin{equation}\label{eq:growth}
    |\hat{q}(x,t)|\leq C' \exp{(b'x^2)} \ \ \text{as} \ \ x\rightarrow \infty,
\end{equation}
for some constants $C'>0$ and  $b' \geq 0$.
Similarly, $\check{q} \in \mathcal{C}^{2,1}\left(\Omega_{T}\right)$ is called a lower solution if it satisfies the reversed inequalities in Eqs. \eqref{eq:uppera}--\eqref{eq:upperd} and the growth condition in Eq. \eqref{eq:growth} with the same or some other smaller constants $C'$ and $b'$.
\end{definition}
\begin{definition}
    The functions $\check{q}$ and $\hat{q}$ are called ordered lower and upper solutions if $\check{q} \leq \hat{q}$ in $\Omega_T.$
\end{definition}
\subsection{ Green's functions}\label{se:greenfun}
The following theorem is helpful to show the existence of  solutions to Eqs. \eqref{eq:Mod1a}--\eqref{eq:Mod1e}, which incorporates a Green's function. Readers may refer to the literature \cite{pao2012} (Chapter 7, Theorem 1.2) for detailed proof and further generalizations. 

\begin{theorem}[\cite{pao2012}, Theorem 7.1.2] \label{th:Main}
    Let $h(x, t)$ and  $u_0(x)$ be continuous functions and satisfy
    \begin{equation}\label{eq:grocon}
        |h(x,t)|\leq A_0 \exp{(b_0 x^2)} \ \text{and} \ |u_0(x)| \leq A_1 \exp{(b_1 x^2)}  \ \text{as} \ |x| \rightarrow \infty
    \end{equation}
    for some constants $A_0, A_1>0$ and $b\geq0$. Then the function 
    \begin{align}
    u(x,t)=\int_{0}^{\infty}G(x,t;\xi, 0)u_0(\xi) \mathrm{d}\xi+\int_{0}^{t}G(x,t;0,s)h(0, s)\mathrm{d}s,
\end{align}
 where
 
\begin{align}
G(x, t ;\xi, s)= & \frac{\exp (-t)}{\sqrt{4\pi(t-s)}}
\left[\exp \left(\frac{-(x-\xi)^2}{4(t-s)}\right)+\exp \left(\frac{-(x+\xi)^2}{4(t-s)}\right)\right]
\end{align}
 is the  solution to  the following system 
\begin{align}
   \nonumber u_t&=u_{xx} - u \ \ \ \forall (x,t) \in \Omega_T, \\
   \nonumber u_x(x,t)&=h(x,t) \ \ \ \forall t \in [0, T] \quad \text{and} \quad x\in \partial \Omega, \\
    \nonumber u(x,0)&=u_0(x) \ \ \ \forall x \ \in \Omega, 
\end{align}
where $T< \frac{1}{4b}$. Moreover, there exist constants $A_2, b_2$ such that
\begin{equation}
    |u(x,t)|\leq A_2 \exp{(b_2x^2)} \quad as \ x\rightarrow \infty.
\end{equation}
\end{theorem}
When $h(x,t)$ and $u_0(x)$ are uniformly bounded, the growth condition Eq. \eqref{eq:grocon} holds with $b_0=0$, $b_1=0$ and the solution $u(x,t)$ is bounded. Moreover, $u(x,t)\rightarrow 0$ as $x\rightarrow \infty$.

The regularity of the solution $u(x,t)$ is discussed in \cite{pao2012} (Chapter 2 and 7, Lemma 1.1) under the regularity assumptions on $u_0(x)$ and $h(x,t)$. In the subsequent sections, we utilize the following properties of  Green's functions to facilitate further analysis. For fixed  $t>s>0$, we have 
    \begin{align}\label{eq:Greeid}
        \int_{0}^{\infty}G(x,t; \xi, 0) \mathrm{d}\xi \leq K < \infty \ \text{and} \ \int_{0}^{t}G(x,t; 0, s) \mathrm{d}s \leq C t\exp{(-t)},
    \end{align}
for some constants $K, C>0$.  The inequality above is derived in Appendix~\ref{ap:ap4}. For broader classes of inequalities, readers are encouraged to consult \cite{pao2007, evans2022,friedman2008}.
\subsection{Stability}
 A fundamental question in this problem is whether, as time progresses, the density function within the system stays close to a steady state and whether it converges to this steady state as $t\rightarrow \infty$. This raises the issue of  stability and asymptotic stability of a steady state solution, along with the determination of its region of stability.
\begin{definition}
    A steady state solution $q_s(x)$ of Eqs. \eqref{eq:Mod1a}--\eqref{eq:Mod1e} is said to be stable if given any constant $\epsilon >0$ there exists a $\delta >0$ such that 
    \begin{equation}\label{eq:sta}
        |q(x,t)-q_s(x)|\leq \epsilon \quad \forall (x,t) \in \Omega_T \  \ \text{whenever} \ |q_s(x)-f(x)|\leq \delta \quad \forall x \in  \ \Omega.
    \end{equation}
    If, in addition, 
    \begin{equation}\label{eq:asysta}
        \lim_{t\rightarrow \infty}|q(x,t)-q_s(x)|=0 \quad \forall x \in \Omega.
    \end{equation}
    Then $q_s(x)$ is said to be asymptotically stable. 
\end{definition}

\begin{definition}
    The set of initial functions $f(x)$ whose corresponding solutions $q(x,t)$ satisfy Eqs. \eqref{eq:sta}--\eqref{eq:asysta} is called stability region of $q_s(x)$. If this is true for
all the initial functions, then $q_s(x)$ is said to be globally asymptotically stable.
\end{definition}

\subsection{Hopf bifurcations }

The behavior of a solution is significantly influenced by the parameters of the system, making it crucial to determine the conditions under which these parameters dictate the long-term dynamics. In systems with delay, time delay often induces Hopf bifurcations near the steady state solution \cite{li2024,hui2022}. To understand the stability of steady-state solutions in such nonlinear problems, the method of linearized stability is commonly employed \cite{guidotti1997,li2024, hui2022}. This approach involves examining the spectrum of the associated linearized eigenvalue problem to predict and characterize the behavior of the steady state solutions. A critical condition for the occurrence of a Hopf bifurcation in systems with a time delay $\tau>0$ is the presence of a pair of complex conjugate eigenvalues crossing the imaginary axis of the complex plane. This phenomenon signifies a transition in the stability of the steady state solutions, where the real part of the eigenvalues changes sign, leading to the emergence of oscillatory behavior \cite{hui2022}.

\section{Main results}
In this section, we establish the following results for Eqs. \eqref{eq:Mod1a}--\eqref{eq:Mod1e}:

\begin{itemize} \item The existence of a positive, bounded solution to Eqs. \eqref{eq:Mod1a}--\eqref{eq:Mod1e}
 with $\tau=0$, given a specific class of initial conditions (see Hypothesis~\ref{hy:ini}) 
 \item The existence of a bounded solution to Eqs. \eqref{eq:Mod1a}--\eqref{eq:Mod1e} with $\tau>0$ \item The global uniqueness of the solution to Eqs. \eqref{eq:Mod1a}--\eqref{eq:Mod1e}
 subject to assumptions in Hypothesis~\ref{hy:uni}
 \item The Hopf bifurcation analysis by treating $\tau$ as a bifurcation parameter
 \item Numerical examples
 \end{itemize}

 \subsection{Existence and uniqueness}
First, we demonstrate the existence of positive, bounded solutions to Eqs. \eqref{eq:Mod1a}--\eqref{eq:Mod1e} in the absence of time delay

\begin{align}
   \label{eq:ModNodela} q_t&=q_{xx} - q, \ (x, t) \in \Omega_T, \\
    \label{eq:ModNodelb} q(\infty,t)&=0,  \ t\in[0,T], \\
   \label{eq:ModNodelc} q_x(0,t)&=\frac{-\alpha}{1+[q(0,t)]^m},  \ \ t\in[0,T], \\
    \label{eq:ModNodeld} q(x,0)&=f(x), \ x\in \Omega, 
\end{align}
for any $0<T<\infty$. 

In practice, the initial concentration is positive near the cell membrane and decays in the far field. This phenomenon is true in many time evolution problems. With this observation,  we assume $f(x)$ as in Hypothesis~\ref{hy:ini}, which helps us define a positive lower solution.  
\begin{hypothesis}\label{hy:ini} For given parameters $\alpha$, $m$ and  initial concentration $f(x)$, there exists $\gamma\geq 2$, $\beta>1$ and $\zeta\geq \sqrt{\beta \gamma (\gamma-1) +1}$ such that $c\exp{(-\zeta x-\beta x^\gamma)} \leq f(x) \ \forall x \in \Omega,$ where $c+c^{m+1}=\alpha$.
\end{hypothesis}

In many cases \cite{pao2002,pao2012}, the zero function is a natural candidate for a lower solution if $f(x)\geq 0$. However, in this particular problem, the negative boundary condition in Eq. \eqref{eq:ModNodelc} eliminates the zero function as a viable lower solution. This necessitates identifying an alternative function that satisfies the required conditions for a lower solution under the given boundary constraints.

The following propositions are straightforward to verify (See Appendix~\ref{ap:ap5}) 
\begin{proposition}\label{pr:lusol1}
    Assume that Hypotheses \ref{hy:gen} and \ref{hy:ini} hold. Then, 
    \begin{equation}
        \check{q}(x,t)=c \exp{(-\zeta x-\beta x^\gamma)} 
    \end{equation}
    is a lower solution to Eqs. \eqref{eq:ModNodela}--\eqref{eq:ModNodeld}, where $c, \zeta, \beta$ and $\gamma$ are defined in Hypothesis~\ref{hy:ini}.  And, 
    \begin{equation}
        \hat{q}(x,t)=M, \quad \text{where} \quad M:=\sup_{x\in \Omega} f(x)
    \end{equation}
    is an upper solution to Eqs. \eqref{eq:ModNodela}--\eqref{eq:ModNodeld}. Moreover,  $\check{q}$ and $\hat{q}$ are ordered lower and upper solutions, i.e. $\check{q}\leq \hat{q}$.
\end{proposition}
\begin{proposition}\label{pr:lusol2}
 The lower and upper solution satisfies the growth condition 
\begin{align}
    |\check{q}(x,t)|&\leq c \exp{(bx^2)}, \\
    |\hat{q}(x,t)|&\leq M \exp{(bx^2)},
\end{align}
 for any constant $b\geq 0$.   
\end{proposition}
Let $\overline{q}^{(0)}=\hat{q}$.  For $k\geq0$, define the sequence $\{\overline{q}^{(k)}\}$ recursively by solving the following system 
\begin{align}
   \label{eq:Iterproa} \overline{q}_t^{(k+1)}&=\overline{q}^{(k+1)}_{xx} - \overline{q}^{(k+1)}, \  (x, t) \in \Omega_T, \\
    \label{eq:Iterprob} \overline{q}^{(k+1)}(\infty,t)&=0, \ t\in[0,T], \\
   \label{eq:Iterproc} \overline{q}^{(k+1)}_x(0,t)&=\frac{-\alpha}{1+[\overline{q}^{(k)}(0,t)]^m}, \ t\in[0,T], \\
    \label{eq:Iterprod} \overline{q}^{(k+1)}(x,0)&=f(x) \ x\in \Omega, 
\end{align}
Similarly, let $\underline{q}^{(0)}=\check{q}$. Define the sequence $\{\underline{q}^{(k)}\}$ using the same recursive system. Under the assumptions in Hypothesis~\ref{hy:gen} and \ref{hy:ini}, 
  Eqs. \eqref{eq:Iterproa}--\eqref{eq:Iterprod} satisfy the hypotheses of Theorem \ref{th:Main} (See Section~\ref{se:greenfun}).  Therefore, the sequences $\{\overline{q}^{(k)}\}$ and $\{\underline{q}^{(k)}\}$ exists and are well defined. In  Lemma \ref{le:monot}, we demonstrate that the sequences $\{\underline{q}^{(k)}\}$ and $\{\overline{q}^{(k)}\}$ generated by the above iterations are nondecreasing and nonincreasing, respectively.
  \begin{lemma}\label{le:monot}
     Let the sequence $\{\underline{q}^{(k)}\}$ and $\{\overline{q}^{(k)}\}$ be generated from Eqs. \eqref{eq:Iterproa}--\eqref{eq:Iterprod}. Then, the sequence $\{\underline{q}^{(k)}\}$ and $\{\overline{q}^{(k)}\}$ are nondecreasing and nonincreasing, respectively.  For each $k$, $\underline{q}^{(k)}$ and $\overline{q}^{(k)}$ are lower and upper solution of Eqs. \eqref{eq:ModNodela}--\eqref{eq:ModNodeld} and satisfy $\underline{q}^{(k)}\leq \overline{q}^{(k)}$. Moreover, the sequences $\{\underline{q}^{(k)}\}$ and $\{\overline{q}^{(k)}\}$ converge pointwise to limits $\underline{q}$ and $\overline{q}$, respectively, and 
    \begin{equation*}
        \check{q}\leq \underline{q}\leq \overline{q}\leq \hat{q}.
    \end{equation*}
\end{lemma}

\begin{proof}
    It is clear that $g(x,t, q):=\frac{-\alpha}{1+ [q(x,t)]^m}$ is continuous on $\overline{\Omega}_T \times \langle \check{q}, \hat{q} \rangle$ and satisfies 
\begin{align}
  \label{eq:lipg}  |g(x,t,q_1)-g(x,t,q_2)|&\leq \Gamma |q_1-q_2| \quad \text{for} \quad \check{q}\leq q_2 \leq q_1 \leq \hat{q},  \\
  \label{eq:growg} |g(x,t,q)|&\leq \alpha  \quad \text{for}  \quad \check{q}\leq q \leq \hat{q},
\end{align}
where $\Gamma=\sup\bigl\{\big|\frac{\partial g}{\partial q}\big|\bigr\}$ in $\check{q} \leq q \leq \hat{q}$ for fixed $(x,t)$. \\
Eqs. \eqref{eq:Iterproa}--\eqref{eq:Iterprod}, together with Eqs. \eqref{eq:lipg}--\eqref{eq:growg}, satisfy the hypotheses of Lemma 3.3 in Chapter 7 of \cite{pao2012}, from which the results follow. For more detail, the reader may refer to Appendix~\ref{ap:ap2}.
\end{proof}

\begin{remark}\label{rm:uni}
    Every solution is an upper and lower solution. If $q$ is a solution in $\langle \check{q}, \hat{q} \rangle$, then $q$ and $\check{q}$ are upper and lower solutions, respectively. By setting $\overline{q}^{(0)}=q$ as initial iteration in Eqs. \eqref{eq:Iterproa}--\eqref{eq:Iterprod}, we get $\{\overline{q}^{(k)}=q\}$ for every $k$.
   From the Lemma \ref{le:monot}, it implies that $\underline{q}\leq q$. Similarly, by considering $q$ and $\hat{q}$ as lower and upper solutions,  then $\overline{q}\geq q$. 
\end{remark}
\begin{theorem}\label{th:ExUnNde}
    Assume Hypothesis~\ref{hy:gen} and \ref{hy:ini} hold. Let $\check{q}$ and $\hat{q}$ be lower and upper solutions of Eqs. \eqref{eq:ModNodela}--\eqref{eq:ModNodeld}. Then  Eqs. \eqref{eq:ModNodela}--\eqref{eq:ModNodeld} has a unique solution in $\langle \check{q}, \hat{q} \rangle$.
\end{theorem} 
\begin{proof}
Using integral representation, the solution  $q^{(k)}$ of  Eqs. \eqref{eq:Iterproa}--\eqref{eq:Iterprod} is 
\begin{equation}
   \label{eq:intesol} q^{(k)}(x,t)=\int_{0}^{\infty}G(x,t;\xi, 0)f(\xi) \mathrm{d}\xi+\int_{0}^{t}G(x,t;0,s)g(0,s,q^{(k-1)}(0,s))\mathrm{d}s,
\end{equation}
where
\begin{align*}
G(x, t ;\xi, s)= & \frac{\exp (-t)}{\sqrt{4\pi(t-s)}}
\left[\exp \left(\frac{-(x-\xi)^2}{4(t-s)}\right)+\exp \left(\frac{-(x+\xi)^2}{4(t-s)}\right)\right],
\end{align*}
and $\{q^{(k)}\}$ is either $\{\overline{q}^{(k)}\}$ or $\{\underline{q}^{(k)}\}$. \\
Since \( g(x, t, q) \) is Lipschitz continuous and bounded by \( \alpha \) in \( \langle \check{q}, \hat{q} \rangle \), it follows that
\begin{equation*}
    |G(x, t; 0, s)\, g(0, s, q^{(k-1)})| \leq \alpha\, |G(x, t; 0, s)|.
\end{equation*}
Clearly, for each fixed $t>s>0$ , $G(x,t; 0, s)$ is integrable. Letting $k\rightarrow \infty$ in Eq. \eqref{eq:intesol}  and applying the dominated convergence theorem shows that the limit $q$ of $q^{(k)}$ satisfies
\begin{align}\label{eq:sol1}
    q(x,t)=\int_{0}^{\infty}G(x,t;\xi, 0)f(\xi) \mathrm{d}\xi+\int_{0}^{t}G(x,t;0,s)g(0,s,q(0,s))\mathrm{d}s.
\end{align} 
From  Theorem \ref{th:Main} and regularity arguments in \cite{pao2012} (Lemma 1.1, Chapter 7),  $q(x,t)$ is the solution of  Eqs. \eqref{eq:ModNodela}--\eqref{eq:ModNodeld}.  \\
From Remark \ref{rm:uni}, all solutions within $\langle \check{q}, \hat{q} \rangle$ lie between $\underline{q}$ and $\overline{q}$. To establish the uniqueness of the solution in $\langle \check{q}, \hat{q} \rangle$, it is sufficient to demonstrate that $\overline{q} = \underline{q}$. Let $w = \overline{q} - \underline{q}$. Then,  
\begin{align}
   \nonumber |w(x,t)|&=\Big|\int_{0}^{t}G(x,t;0,s)\left(g(0,s,\overline{q}(0,s))-g(0,s,\underline{q}(0,s))\right)\mathrm{d}s\Big|,\\
   \label{eq:uniqeq} &\leq \Gamma \Big|\int_{0}^{t}G(x,t;0,s)w(0,s)\mathrm{d}s\Big|.
\end{align} 
 For each fixed $t>s>0$, we get 
\begin{align*}
    G(x,t;0,s)\leq \frac{1}{\sqrt{\pi (t-s)}}.\\
\end{align*}
For fixed $t>0$, define 
\begin{equation}
    \|w\|_t=\sup\big\{|w(\xi,s)|:  (\xi, s) \in \Omega \times [0, t]\big\}.
\end{equation}
Then Eq.\eqref{eq:uniqeq} implies that
\begin{align*}
    |w(x,t)|&\leq \frac{\Gamma \|w\|_t}{\sqrt{\pi}} \Big|\int_{0}^{t}\frac{1}{\sqrt{t-s}} \mathrm{d}s\Big|\\
    |w(x,t)|&\leq K_1t^{\frac{1}{2}}\|w\|_t \quad \forall x \in \Omega,
\end{align*}
where $K_1=\frac{2\Gamma}{\sqrt{\pi}}$ is constant independent of $(x,t)$. \\
If $\|w\|_{t}>0$, choose $t_1>0$ such that $t_1^{\frac{1}{2}}K_1<\epsilon<1$. Since $\|w\|_t$ is a nondecreasing function of $t$, 
the above inequality  gives 
\begin{equation*}
    |w(x,t)|< \epsilon\|w\|_{t_1} <\|w\|_{t_1}\quad \forall (x,t) \in \Omega \times [0, t_1].
\end{equation*}
It implies that, $\epsilon \|w\|_{t_1}$ is upper bound for $|w(x,t)|$ in $\Omega \times [0,t_1]$, which contradicts the definition of $\|w\|_{t_1}$. Therefore, $\|w\|_{t_1}=0$, which shows that $\underline{q}=\overline{q}$ in $\Omega \times [0, t_1]$. Using $q(x,t_1)$ as initial condition in the domain $\Omega\times(t_1,T]$, a continuation of above argument leads to the conclusion $\underline{q}=\overline{q}$ in $\Omega_T$.
\end{proof}
In the following Theorem, we discuss the existence of a solution to a time delay system.
\begin{theorem}\label{th:UETM}
    Assume Hypothesis~\ref{hy:gen} holds. Then Eqs. \eqref{eq:Mod1a}--\eqref{eq:Mod1e} with $\tau>0$  has a solution. 
\end{theorem}
\begin{proof}
To show the existence of the solution to time delay Eqs. \eqref{eq:Mod1a}--\eqref{eq:Mod1e} with $\tau>0$, we consider the problem in subdomain $\Omega_1=[0,\tau]\times \Omega$. Since $q(0,t-\tau)$ is know in $[0,\tau]$, there exists a solution in $\Omega_1$ (see Theorem \ref{th:Main}). Knowing the solution in $\Omega_1$, a ladder argument ensures the existence of the solution to Eqs. \eqref{eq:Mod1a}--\eqref{eq:Mod1e} on $\Omega_m=[0, m\tau]\times \Omega$ for every positive integer $m=1,2,3...$ . 
\end{proof}
In Theorem \ref{th:ExUnNde}, we demonstrated the uniqueness of solutions in $\langle \check{q}, \hat{q} \rangle$. In the following theorem, we show the global uniqueness of the solution in $\mathcal{C}^{2,1}(\Omega_T)$ under the assumptions of Hypothesis~\ref{hy:uni}. 
    \begin{hypothesis}
        \label{hy:uni}
    (Global uniqueness)
    Let $q(x,t)$ be the solution of Eqs. \eqref{eq:Mod1a}--\eqref{eq:Mod1e}. Where
    \begin{itemize}
    \item $q(x,t)\geq 0$;
        \item $\frac{\partial q}{\partial x} \rightarrow 0$ as $x \rightarrow \infty$; and 
        \item $q(x,t), q_x(x,t) \in L^2(\Omega)$ for each $t$.
    \end{itemize}
    \end{hypothesis} 
    \begin{theorem}\label{th:Glob}
        Assume Hypothesis~\ref{hy:gen}, \ref{hy:ini} and \ref{hy:uni} hold.  Then Eqs. \eqref{eq:Mod1a}--\eqref{eq:Mod1e} have a unique solution in $C^{2,1}(\Omega_T)$.
    \end{theorem}
    \begin{proof}
        We first prove the result for the case $\tau =0$, and then extend the proof for $\tau>0$ using the ladder argument.
    Suppose $q_1$ and $q_2$ are two solutions to Eqs. \eqref{eq:Mod1a}--\eqref{eq:Mod1e}, then $w=q_1-q_2$ satisfies 
    \begin{align}
   \label{eq:globunia} w_t&=w_{xx} - w,\   (x, t) \in \Omega_T, \\
    \label{eq:globunib} w(\infty,t)&=0, \ t\in[0,T], \\
   \label{eq:globunic} w_x(0,t)&=\frac{-\alpha}{1+[q_1(0,t)]^m}+\frac{\alpha}{1+[q_2(0,t)]^m}, \ t\in[0,T], \\
    \label{eq:globunid} w(x,0)&=0,  \ x\in \Omega. 
\end{align}
       From the natural estimation of Eqs. \eqref{eq:globunia}--\eqref{eq:globunid}
       \begin{align*}
           \frac{1}{2}\frac{\mathrm{d}}{\mathrm{d}t}\int_{\Omega}|w(x,t)|^2\mathrm{d}x&=-\int_{\Omega}|w_x(x,t)|^2\mathrm{d}x-\int_{\Omega}|w(x,t)|^2\mathrm{d}x\\&+\left(\frac{\alpha}{1+[q_1(0,t)]^m}-\frac{\alpha}{1+[q_2(0,t)]^m}\right)\left(q_1(0,t)-q_2(0,t)\right),\\
           &\leq \left(\frac{\alpha}{1+[q_1(0,t)]^m}-\frac{\alpha}{1+[q_2(0,t)]^m}\right)\left(q_1(0,t)-q_2(0,t)\right),\\
           &= \frac{\alpha \left([q_2(0,t)]^m-[q_1(0,t)]^m\right)}{\left(1+[q_1(0,t)]^m\right)\left(1+[q_2(0,t)]^m\right)} \left(q_1(0,t)-q_2(0,t)\right),\\
           &\leq 0.
       \end{align*}
       It implies that, 
       \begin{equation*}
           \int_{\Omega}|w(x,t)|^2\mathrm{d}x=0.
       \end{equation*}
       Hence, $q_1=q_2$ $\forall (x,t) \in \Omega_T$.
       
        For $\tau>0$, $q_1(0, t-\tau)=q_2(0,t-\tau)=h(t-\tau)$ in $[0, \tau]$, it implies that $w_x(0, t)=0 \ \forall t \in [0,\tau]$. Using similar arguments as above, the solution is unique in $\Omega \times [0,\tau]$. Since  $q_1(0, t-\tau)=q_2(0,t-\tau)$ in $[0, 2\tau]$, there exists a unique solution in $\Omega \times [0,2\tau]$.  The ladder argument thus shows  the uniqueness of solutions in $\Omega \times [0,m\tau], \ \ m=1,2,3...$ 
    \end{proof}
        \begin{remark}
        Under the assumptions in Hypothesis \ref{hy:gen}, \ref{hy:ini} and \ref{hy:uni}, there exists a unique solution in $\Omega \times [0, T]$ for the  following model 
        \begin{align}
   \nonumber q_t&=q_{xx} - q, \\
    \nonumber q(\infty,t)&=0, \\
   \nonumber q_x(0,t)&=\frac{-\alpha}{1+q(0,t-\tau)^m} -\frac{\beta}{1+q(0,t)^m},\\
    \nonumber q(x,0)&=f(x), \\
   \nonumber  q(0,t)&=h(t), \quad t \in [-\tau,0].
\end{align}
    \end{remark}

 \subsection{Asymptotic behavior of solutions}
From a practical perspective, studying the long-term behavior of the solution is essential. The following lemma demonstrates that the solution remains bounded as $t \rightarrow \infty$. Moreover, since Eqs. \eqref{eq:Mod1a}--\eqref{eq:Mod1e} depend on parameters $\alpha, m$ and $\tau$, the stability of the steady state solution is influenced by these parameters. Time delay is a key parameter in the time evolution problem. Thus, we treat the time delay as a bifurcation parameter and examine its effect on the stability of the steady state solution.

The equivalent steady state problem of the Eqs.    \eqref{eq:Mod1a}--\eqref{eq:Mod1e} is
\begin{align*}
   \nonumber  q_{xx} - q&=0,  \  x \in \Omega, \\
    \nonumber q(\infty)&=0,   \\
   \label{eq:steNodel} q_x(0)&=\frac{-\alpha}{1+[q(0)]^m}.
\end{align*}
There exists a unique $c$ in $(0, \alpha)$ such that $c+c^{m+1}=\alpha$ and the steady state solution is given as 
\begin{equation}
    q_s(x)=c\exp{(-x)} \quad \forall x \in \Omega.
\end{equation}
\begin{proposition} 
    Let $q(x,t)$ be the solution of Eqs. \eqref{eq:Mod1a}--\eqref{eq:Mod1e} given in   Eq. \eqref{eq:sol1}. Then $\exists K_1 <\infty$, such that 
    \begin{equation*} \limsup_{t\rightarrow\infty}|q(x,t)| \leq K_1, \ \forall x\in \Omega.
    \end{equation*}
    \end{proposition}
    \begin{proof} Since $|f(x)|\leq M$ and $|g(0,t,q)|\leq \alpha$ are bounded functions, we get
        \begin{align*}
            |q(x,t)| \leq M \int_{0}^{\infty}|G(x,t;\xi, 0)| \mathrm{d}\xi+\alpha \int_{0}^{t}|G(x,t;0,s)|\mathrm{d}s.
        \end{align*} 
        From the estimation Eq. \eqref{eq:Greeid} and letting $t\rightarrow \infty$, above inequality gives
        \begin{equation*}
          \limsup_{t \rightarrow \infty} |q(x,t)|\leq K_1,
        \end{equation*}
        where $K_1=MK$.
    \end{proof}
\begin{proposition}  Assume Hypothesis \ref{hy:gen} and \ref{hy:uni} hold.
 Let $q_s(x)$ be the steady state solution of Eqs. \eqref{eq:ModNodela}--\eqref{eq:ModNodeld}. If $c\exp{(-x)}\leq f(x)\leq M $. Then 
 \begin{equation*}
     |q(x,t)-q_s(x)|\leq M \quad \text{for} \quad t\geq 0 \quad \text{and} \quad x \in \Omega.
 \end{equation*}
\end{proposition}
\begin{proof}
Since $q_s(x)$ and $M$ are ordered lower and upper solutions of Eqs. \eqref{eq:ModNodela}--\eqref{eq:ModNodeld},  there exists a unique solution $q(x,t)$ such that 
\begin{equation*}
    q_s(x)-M\leq q_s(x)\leq q(x,t)\leq M \leq q_s(x)+M.
\end{equation*}
 It implies 
 \begin{equation*}
     |q(x,t)-q_s(x)|\leq M  \quad \text{for}\quad t\geq 0 \quad \text{and} \quad x\in \Omega.
 \end{equation*}
\end{proof}

We now consider Hopf bifurcations. For $\tau > 0$, the function $q_s(x) = c\exp(-x)$ is the unique positive steady state solution for fixed values of $m$ and $\alpha$. To analyze the stability characteristics of this solution, we linearize Eqs. \eqref{eq:Mod1a}--\eqref{eq:Mod1e} using  
\begin{equation} \label{eq:purt}
q(x,t)=c\exp{(-x)}+v(x,t),
\end{equation}
This approach allows us to systematically assess the conditions under which the steady state remains stable in response to variations in the bifurcation parameter $\tau$. By substituting Eq. \eqref{eq:purt} into the system and using a Taylor series expansion for the nonlinear term in Eqs. \eqref{eq:Mod1a}--\eqref{eq:Mod1d}, we obtain 
\begin{align}
   \label{eq:lina} v_t-v_{xx}+v&=0, \ t\geq 0, x\in \Omega, \\
  \label{eq:linb} v(\infty, t)&=0, \ t\geq 0,\\
  \label{eq:linc} v_x(0, t)&=Qv(0, t-\tau), t\geq0,
\end{align}
where $Q=\frac{m \alpha c^{m-1}}{(1+c^{m})^2}$.

We use the principle of linearized stability to analyze the stability properties of the steady-state solution \cite{guidotti1997,hui2022,li2024}. The corresponding eigenvalue problem is formulated \cite{andrev2007,memory1989} by assuming a solution of the form $v(x,t)=\exp{(\lambda t)}\phi(x)$, which leads to the following system 
\begin{align}
    \label{eq:eiga} \phi_{xx}&=(1+\lambda)\phi \quad x\in \Omega, \\
    \label{eq:eigb} \phi(\infty)&=0, \\
   \label{eq:eigc} \phi_x(0)&=Q\exp{(-\lambda \tau)} \phi(0).
\end{align}
To compute the eigenvalues of Eqs. \eqref{eq:eiga}--\eqref{eq:eigc}, we use 
\begin{equation} \label{eq:eigfun}
    \phi(x)=A\exp{(-\rho x)}+B\exp{(\rho x)},
\end{equation}
which solves Eq. \eqref{eq:eiga}. Where $A,B,\rho \in \mathbb{C}$ and $\lambda=\rho^2-1$. \\
By applying Eqs. \eqref{eq:eigb}--\eqref{eq:eigc} in Eq. \eqref{eq:eigfun} , we get 
\begin{align}
   \label{eq:Fp} F_p(\rho, \tau)=\rho+Q\exp{(-(\rho^2-1)\tau)}&=0 \quad \text{if} \quad Re(\rho)>0,\\
   \label{eq:Fn} F_n(\rho, \tau)=\rho-Q\exp{(-(\rho^2-1)\tau)}&=0 \quad  \text{if} \quad Re(\rho)<0,
\end{align}
where $Re$ denotes the real part. Here, we demonstrate the case $Re(\rho)>0$, and the other case follows the same steps. To analyze the zeros of the function $F_p(\rho, \tau)=0$, we decompose it into its real and imaginary components. By expressing $\rho$ in terms of its real and imaginary parts, $\rho=a+ib$, the function $F_p(\rho, \tau)=0$ can be rewritten accordingly. Solving this system allows us to locate the zeros of $F_p(\rho, \tau)=0$ in the complex plane. We have
\begin{align}
  \label{eq:repa}  a+Q\exp{(-(a^2-b^2-1)\tau)}\cos{(2ab\tau)}&=0,\\
   \label{eq:impa} b-Q\exp{(-(a^2-b^2-1)\tau)}\sin{(2ab\tau)}&=0.
\end{align}
If $0\leq Q<1$, from Eq. \eqref{eq:repa} and Eq. \eqref{eq:impa} we have 
\begin{align}
    a^2-b^2-1&=Q^2\exp{(-2(a^2-b^2-1)\tau)}\cos{(4ab\tau)}-1, \\
    a^2-b^2-1&<\exp{(-2(a^2-b^2-1)\tau)}-1.
\end{align}
 This implies that  $a^2-b^2-1<0$ for any $\tau>0$. Since $Re(\lambda)<0$, the steady state is locally stable. Consequently, a Hopf bifurcation is anticipated when $Q\geq 1$. 

Next, we establish the eigenvalues of Eqs. \eqref{eq:eiga}--\eqref{eq:eigc} cross the imaginary axis for some $\tau>0$, which is a necessary condition for the occurrence of a Hopf bifurcation at fixed values of $m$ and $\alpha$. Specifically, Eqs. \eqref{eq:eiga}--\eqref{eq:eigc} possesses a purely imaginary eigenvalue $\lambda=\pm i\mu \ (\mu\neq 0)$ for certain values of $\tau>0$ if and only if the following condition is satisfied
\begin{equation}
    \label{eq:hocon} Re(\lambda) = a^2-b^2-1=0.
\end{equation}
From Eqs. \eqref{eq:repa},  \eqref{eq:impa} and \eqref{eq:hocon}, we get
\begin{align}
   \label{eq:hosy} a^2+b^2=Q^2, \quad \text{and} \quad \frac{-b}{a}=\tan{(2ab\tau)}.
\end{align}
In the case where  $Re(\rho)<0$, the roots of $F_n(\rho, \tau)=0$ also satisfy   \eqref{eq:hosy}. For $Q>1$, the points of intersection of the curves $a^2+b^2=Q^2$ and $a^2-b^2=1$ are given by  $a=\pm\sqrt{\frac{Q^2+1}{2}}, b=\pm \sqrt{\frac{Q^2-1}{2}}$. It is straightforward to verify that the following equation
\begin{equation}\label{eq:taueq}
    -\frac{\sqrt{Q^2-1}}{\sqrt{Q^2+1}}=\tan{(\tau \sqrt{Q^4-1} )}
\end{equation}
has a root in the interval  $ \left(\frac{\pi}{2\sqrt{Q^4-1}},  \frac{\pi}{\sqrt{Q^4-1}}\right)$, where $\sqrt{Q^4-1}$ is the positive imaginary part of the eigenvalues.  
With this, we are now ready to state the following theorem.
\begin{theorem}\label{th:hobi}
    \begin{enumerate}
        \item Assume $Q<1$, then Eqs. \eqref{eq:eiga}--\eqref{eq:eigb} has only  eigenvalues with negative real part for any $\tau>0$.
        \item Assume $Q>1$, then there exist unique $\tau_0 \in \left(\frac{\pi}{2\sqrt{Q^4-1}},  \frac{\pi}{\sqrt{Q^4-1}}\right)$ such that 
        \begin{equation*}
            \bigl\{\lambda \in \mathbb{C}:\lambda=\rho^2-1, \rho\pm Q\exp{(-(\rho^2-1)\tau_0)}=0 \bigr\}\bigcap i\mathbb{R}=\bigl\{-i\sqrt{Q^4-1}, +i\sqrt{Q^4-1}\bigr\}.
        \end{equation*}
    \end{enumerate}
\end{theorem}
\begin{proof}
    Proof is straightforward from the above discussion.
\end{proof}

\begin{remark}
    \begin{enumerate}
        \item There exist countably many values of $\tau>\tau_0$ that satisfy the second part of the above theorem.
        \item For $0<\tau<\tau_0$, Eqs. \eqref{eq:eiga}--\eqref{eq:eigc} has no purely imaginary eigenvalues.
    \end{enumerate}
\end{remark}
Along with the above result, we need to check that the pair of complex conjugate eigenvalues crosses the imaginary with positive speed at $\tau_0$.
\begin{equation}
  \text{i.e.} \quad  \frac{\mathrm{d}}{\mathrm{d}\tau}Re(\lambda(\tau_0))>0.
\end{equation}
To verify the inequality above, we first note that
\begin{align*}
    \frac{\mathrm{d}\lambda}{\mathrm{d}\tau}&=2\rho \frac{\mathrm{d}\rho}{\mathrm{d}\tau}
\end{align*}
and, 
\begin{equation}
   \label{eq:derre} \frac{\mathrm{d}}{\mathrm{d}\tau}Re(\lambda(\tau_0))=2Re\left(\rho(\tau_0)\frac{\mathrm{d}\rho(\tau_0)}{\mathrm{d}\tau}\right)=2a(\tau_0)Re\left(\frac{\mathrm{d}\rho(\tau_0)}{\mathrm{d}\tau} \right)-2b(\tau_0)Im\left(\frac{\mathrm{d}\rho(\tau_0)}{\mathrm{d}\tau}\right)
\end{equation}
Where $Im$ denotes the imaginary part. The implicit function $F_p(\rho,\tau)=0$ satisfies the conditions of the Implicit Function Theorem at the point $(\rho(\tau_0), \tau_0)$, where $\rho(\tau_0)=a(\tau_0)\pm ib(\tau_0)=\frac{\sqrt{Q^2+1}}{2}\pm i\frac{\sqrt{Q^2-1}}{2}$. The derivative formula of the implicit function gives \cite{guidotti1997}
\begin{equation}
   \label{eq:derim} \frac{\mathrm{d}\rho (\tau_0)}{\mathrm{d}\tau}=-\left(\frac{\partial F_p(\rho(\tau_0),\tau_0)}{\partial \rho}\right)^{-1}\left(\frac{\partial F_p(\rho(\tau_0),\tau_0)}{\partial \tau}\right).
\end{equation}
The following calculations are straightforward \begin{align}
  \label{eq:im1}  \frac{\partial F_p(\rho_0,\tau_0)}{\partial \tau}&=-2Qa_0b_0\biggl(\sin{(2a_0b_0\tau_0)+i\cos(2a_0b_0\tau_0)}\biggr)\\
     \nonumber \frac{\partial F_p(\rho_0,\tau_0)}{\partial \rho}&=\biggl(1-2Q\tau_0\bigl(a_0\cos{(2a_0b_0\tau_0)}+b_0\sin{(2a_0b_0\tau_0)}\bigr)\biggr)\\
     \label{eq:im2}&-2iQ\tau_0\biggl(b_0\cos{(2a_0b_0\tau_0)}-a_0\sin{(2a_0b_0\tau_0)}\biggr)
\end{align} 
By substituting Eq. \eqref{eq:im1} and \eqref{eq:im2} into Eq. \eqref{eq:derim} and then applying Eq. \eqref{eq:derre}, we obtain 
\begin{equation}
    \frac{\mathrm{d}}{\mathrm{d}\tau}Re(\lambda(\tau_0))=\frac{4Qa_0b_0}{D_1^2+D_2^2}\left(a_0\sin{(2a_0b_0\tau_0)-b_0\cos{(2a_0b_0\tau_0)}} \right),
\end{equation}
where
\begin{align*}
    D_1&=\bigl(1-2Q\tau_0\bigl(a_0\cos{(2a_0b_0\tau_0)}+b_0\sin{(2a_0b_0\tau_0)}\bigr)\bigr),\\
    D_2&=2Q\tau_0\bigl(b_0\cos{(2a_0b_0\tau_0)}-a_0\sin{(2a_0b_0\tau_0)}\bigr).
\end{align*}
For $\tau_0 \in \left(\frac{\pi}{2\sqrt{Q^4-1}},  \frac{\pi}{\sqrt{Q^4-1}}\right)$, $a_0=\sqrt{\frac{Q^2+1}{2}}$ and $b_0=\pm \sqrt{\frac{Q^2-1}{2}}$, we get 
\begin{equation*}
    a_0b_0\left(a_0\sin{(2a_0b_0\tau_0)-b_0\cos{(2a_0b_0\tau_0)}} \right)>0.
\end{equation*}
It implies 
\begin{equation*}
    \frac{\mathrm{d}}{\mathrm{d}\tau}Re(\lambda(\tau_0))>0.
\end{equation*}
From Theorem \ref{th:hobi} and the discussion above, we conclude that a Hopf bifurcation point $\tau_0$ exists if $Q>1$. In the next section, we illustrate these results with a numerical example.  

\subsection{Numerical simulations}
This section provides numerical simulations of Eqs. \eqref{eq:Mod1a}--\eqref{eq:Mod1e} to validate the theoretical results on Hopf bifurcations. To illustrate the different dynamical behaviors predicted by the analysis, we consider four examples, as summarized in Table \ref{tb:nuex}. Examples 1 and 2 correspond to the case $Q<1$, where the steady state solution is expected to be stable, while Examples 3 and 4 correspond to the case $Q>1$, where a Hopf bifurcation occurs, leading to sustained oscillations when time delay exceeds the threshold.

For Examples 1 and 2 $(Q<1)$, numerical solutions were computed for time delays $\tau=5, \tau=10$, and $\tau=15$. The results presented in Figs. \ref{fig:eg1} and \ref{fig:eg2}, corresponding to Examples 1 and 2, validate the stability of the steady state solution. In both examples, the simulations demonstrate that the initial oscillations in the system decay over time, indicating that the system eventually converges to the steady state solution. This behavior aligns with the theoretical prediction that the steady state is stable for $Q<1$ regardless of the delay parameter.

For Examples 3 and 4 $(Q>1)$, the numerical results are presented in Fig. \ref{fig:eg3}  and  \ref{fig:eg4}, respectively, demonstrating the behavior of the system in the presence of a Hopf bifurcation. According to Theorem \ref{th:hobi} and Eq. \eqref{eq:taueq}, the critical time delay for Example 3 is calculated to be approximately $\tau_0\approx1.1$, within the interval $(0.67723, 1.3448)$. Similarly, for Example 4, the Hopf bifurcation occurs at $\tau_0\approx0.12$, within the interval $(0.07601, 0.15203)$. In both examples, the numerical simulations reveal that when the time delay $\tau$ exceeds the critical value $\tau_0$, the solutions exhibit sustained oscillations. Figs. \ref{fig:eg1}--\ref{fig:eg4} represent the protein concentration at the cell membrane, whereas 3D plots (plots of the solution over space and time) are shown in Appendix~\ref{ap:ap3}. We handle the numerics with a discretization as in \cite{miller2024}.  

\vspace{1cm}
\begin{table}[h]
\caption{Numerical examples}
\label{tb:nuex}
\centering
\setlength{\tabcolsep}{13pt}
\begin{tabular}{|p{2cm}|p{1cm}|p{1cm}|p{1cm}|p{2.5cm}|p{4cm}|}
\hline
\textbf{Example} & $m$ & $\alpha$ & $c$ & $Q = \dfrac{m\alpha c^{m-1}}{(1 + c^m)^2}$ & $\tau_0$ \\
\hline
Example 1 & 4 & 0.4 & 0.3909 & 0.0912 & No Hopf bifurcation \\
\hline
Example 2 & 2 & 1.6 & 0.8915 & 0.8856 & No Hopf bifurcation \\
\hline
Example 3 & 4 & 1.5 & 0.9022 & 1.5941 & $\approx 1.1$ \\
\hline
Example 4 & 6 & 5 & 1.2097 & 4.5484 & $\approx 0.12$ \\
\hline
\end{tabular}
\end{table}

\begin{figure}[h]
 \centering  \includegraphics[width=\textwidth]{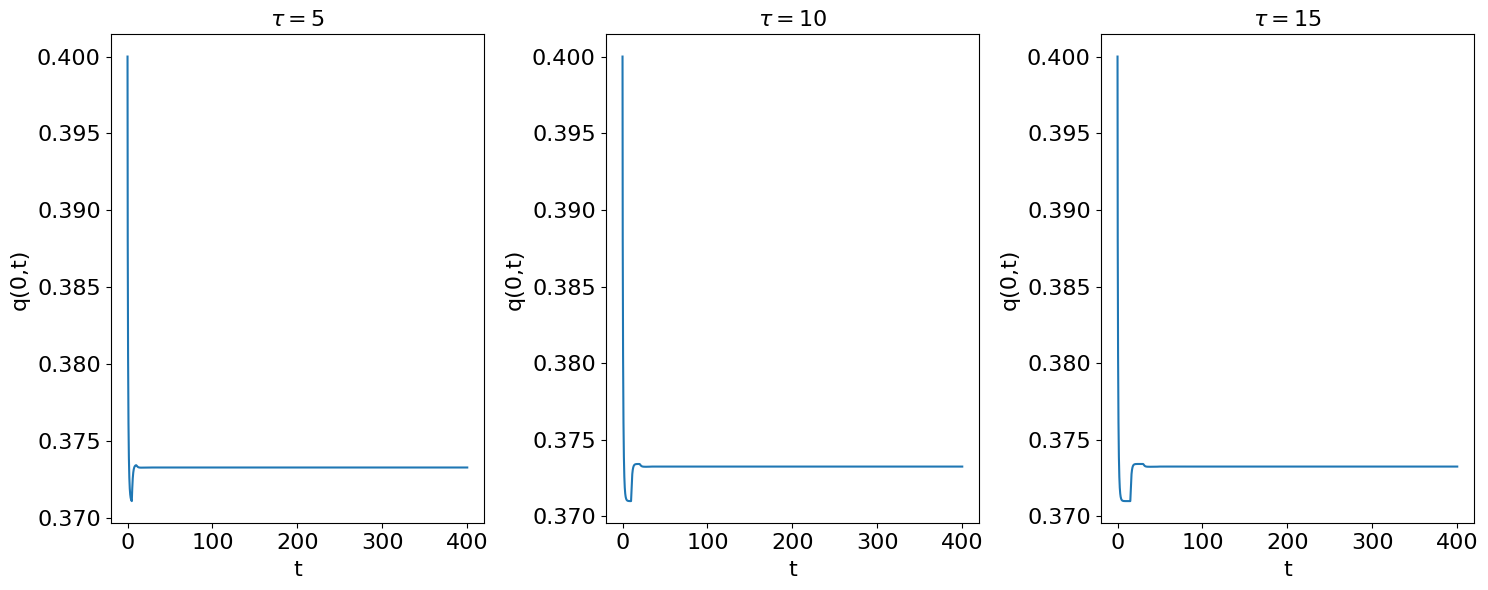}
  \caption{Example 1 (Protein concentration $q(x,t)$ at the cell membrane for $m=4$ and $\alpha=0.4$ )} 
  \label{fig:eg1}  
\end{figure}  

\begin{figure}[h!]
 \centering
\includegraphics[width=\textwidth]{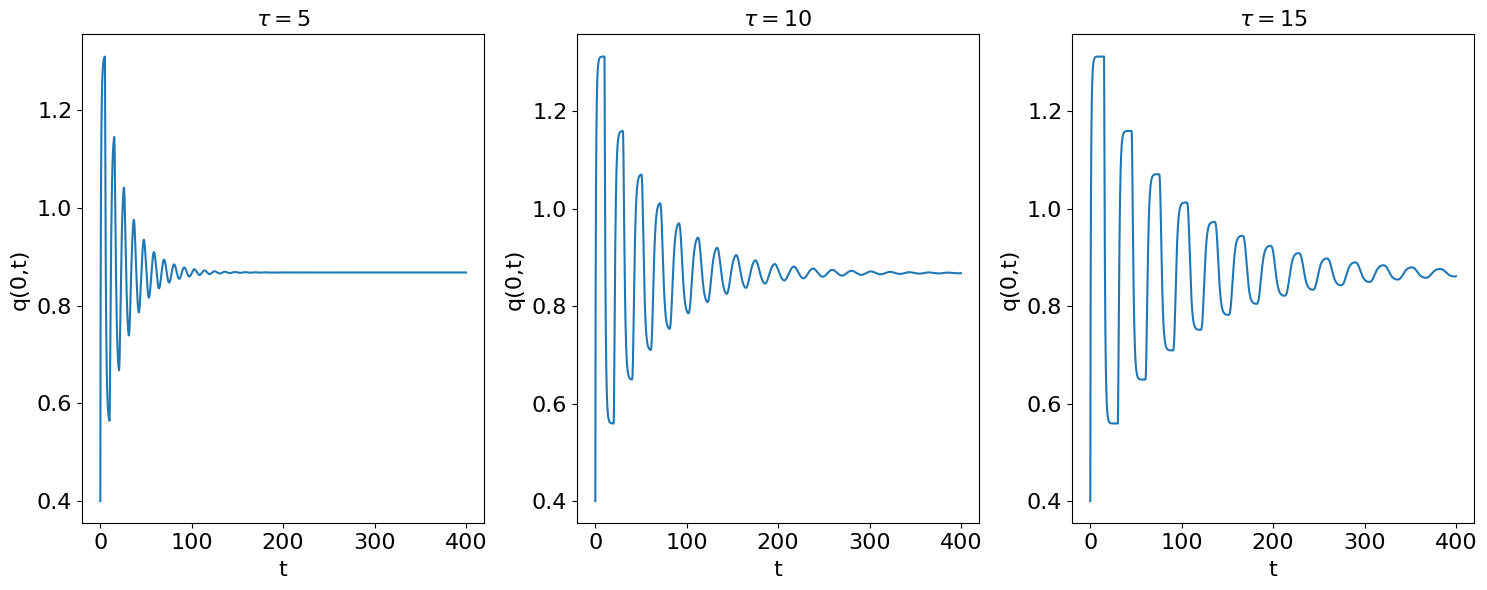}
  \caption{Example 2 (Protein concentration $q(x,t)$ at the cell membrane for $m=2$ and $\alpha=1.6$ )} 
  \label{fig:eg2} 
\end{figure} 

\begin{figure}[h!]
 \centering
\includegraphics[width=\textwidth]{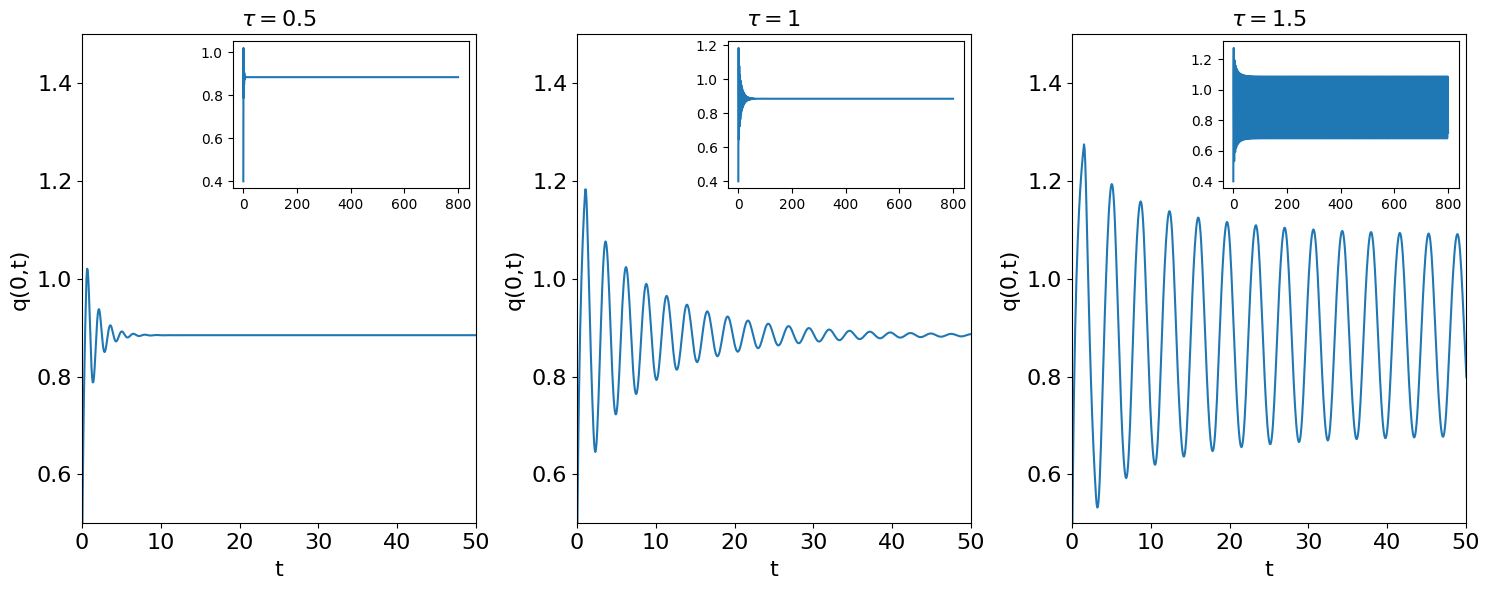}
  \caption{Example 3 (Protein concentration $q(x,t)$ at the cell membrane for $m=4$ and $\alpha=1.5$ )} 
  \label{fig:eg3} 
\end{figure} 

\begin{figure}[h!]
 \centering
\includegraphics[width=\textwidth]{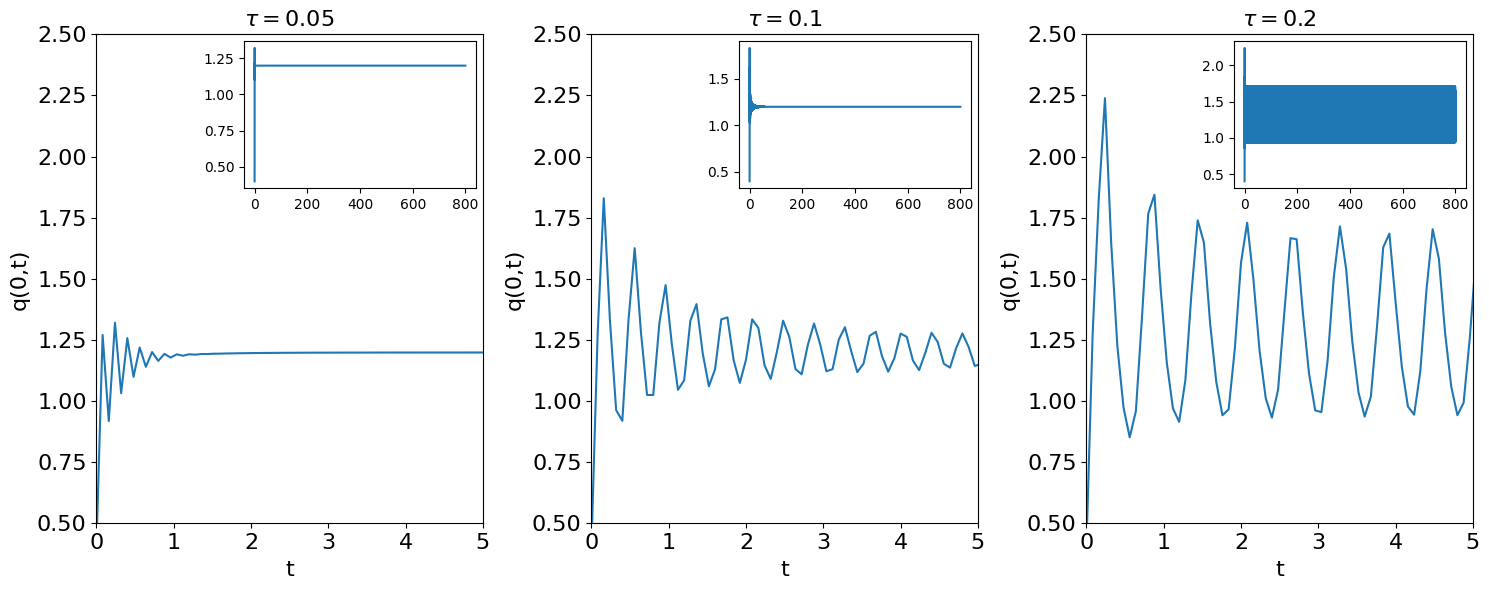}
  \caption{Example 4 (Protein concentration $q(x,t)$ at the cell membrane for $m=6$ and $\alpha=5$ )} 
  \label{fig:eg4}
\end{figure}

\newpage
In the following remark, we discuss the sensitivity of the parameter to stability. The parameter $m$ controls how rapidly the cell switches from maximum flux to zero flux, and $\alpha=\frac{A}{P_0\sqrt{Dk}}$ represents the dimensionless maximum flux parameter. 
\begin{remark}
    \begin{enumerate}
        \item If $\alpha <1$ and $m\gg1$, then $c\sim \alpha$ and $Q\sim \frac{\alpha m \alpha ^{m-1}}{(1+\alpha^m)^2}$ $(Q\rightarrow 0 \quad \text{as} \quad m\rightarrow \infty)$. This indicates that when the boundary flux $\alpha$ is sufficiently small $(\alpha<1)$, an increase in the nonlinearity parameter $m$ does not disrupt the stability of the steady state solution for any $\tau$.
        \item If $\alpha >1$ and $m\gg1$, then $c\sim 1+\frac{\log{(\alpha-1)}}{m}$ and $Q\sim \frac{\alpha-1}{\alpha}m$. This suggests that increasing the nonlinearity parameter $m$ induces a Hopf bifurcation for some  $\tau>0$ when $\alpha>1$.
    \end{enumerate}  
\end{remark}    

\section{Conclusion and future work}
In this study, we investigated a parabolic partial differential equation with a nonlinear, time-delayed boundary condition, focusing on the existence, stability, and oscillatory behavior of its solutions. We demonstrated the existence of a unique positive, bounded solution under a suitable hypothesis, applicable both in the presence and absence of delay. The local asymptotic stability of the steady-state solution was rigorously analyzed through the associated eigenvalue problem, revealing the dependence of solution behavior on the nonlinear parameter $m$, boundary flux $\alpha$, and time delay $\tau$.  For fixed $m$ and $\alpha$, the time delay $\tau$ was treated as a bifurcation parameter to examine the stability of the steady state solution. A critical insight from this analysis is the occurrence of a Hopf bifurcation at a critical delay $\tau_0 \in \left(\frac{\pi}{2\sqrt{Q^4-1}}, \frac{\pi}{\sqrt{Q^4-1}}  \right)$ when $Q>1$. This indicates the transition from stability to oscillatory behavior as $\tau$ crosses the bifurcation threshold. On the other hand, the steady-state solution remains stable for any $\tau>0$ when $Q<1$. We also validated the theoretical results on Hopf bifurcation through numerical examples, illustrating the critical delay values and the emergence of periodic solutions predicted by the analysis.

The model considered in this study is limited to a single cell producing a single protein species, resulting in a scalar PDE. In future work, we could extend our focus to cells producing multiple protein species, which will lead to a system of coupled PDEs rather than a single equation.

As an initial step, we developed the mathematical theory for the model by representing the cell as an interval on the real line, thereby simplifying the domain (interstitial fluid) to a half-space. However, a more realistic approach would involve modeling the cell as a three-dimensional object, with the computational domain defined as the exterior of the cell. Future work could focus on the analysis of multiple cells within a bounded domain (brain) and the development of the corresponding mathematical theory for the resulting system.

Here, we assume the initial condition is continuous, which helps us to prove the existence of a classical solution to our model. In the future, one could consider weaker initial conditions in $L^p$ spaces and establish the existence of solutions in weaker function spaces.

Since our primary objective is to demonstrate the existence of a Hopf bifurcation in the model induced by time delay, we have focused on linearized stability analysis rather than global methods such as Lyapunov functions or comparison principles. One could consider studying the global stability of the model in the future.

\section*{Acknowledgments}
This project has received support from NSF Award DMS \#2316952.

\section*{Appendix}
\section{The Phragman-Lindelof principle\label{ap:ap1}} 

Let $q(x,t)$ satisfy 
\begin{align*}
    q_t-q_{xx}+q&\geq0 \quad \text{in} \quad \Omega_T, \\
    a\frac{\partial q}{\partial x}+bq&\geq 0 \quad \text{in} \quad \partial \Omega \times [0,T],\\
    q(x,0)&\geq 0 \quad \text{in} \quad \Omega,
\end{align*}
where $a, b\geq 0$ and $a+b>0$. If there exists a constant $\delta>0$ such that 
\begin{equation*}
    \limsup_{R\rightarrow{\infty}}\bigl[\exp{(-\delta R^2)}\left(\min\bigl\{q(x,t):0\leq t\leq T, x=R\bigr\}\right)\bigr]\geq 0
\end{equation*}
then $q(x,t)\geq 0$ in $\overline{\Omega}_T$ (See  Theorem 1.3 of Chapter 7 of \cite{pao2012}).

\section{Proof of Lemma \ref{le:monot}}\label{ap:ap2}
\begin{proof}
    \textbf{Step 1.} Here, we prove that the sequence $\{\underline{q}^{(k)}\}$ is nondecreasing and serves as an lower solution. Consider $k=1$ with $\check{q}=\underline{q}^{(0)}$. From Theorem \ref{th:Main}, there exists a solution $\underline{q}^{(1)}$ to Eqs. \eqref{eq:Iterproa}--\eqref{eq:Iterprod} and it satisfies $|\underline{q}^{(1)}(x,t)|\leq M_1\exp{(bx^2)}$ as $x\rightarrow{\infty}$ for some $M_1\geq0$ and $b\geq 0$.  Let $w=\underline{q}^{(1)}-\underline{q}^{(0)}$. Then,
    \begin{align*}
        w_t-w_{xx}+w=\underline{q}_t^{(1)}-\underline{q}^{(1)}_{xx}+\underline{q}^{(1)}-\left(\underline{q}^{(0)}_t-\underline{q}^{(0)}_{xx}+\underline{q}^{(0)}\right) \quad \text{in} \quad \Omega_T.
    \end{align*}
From definition of lower solution and the iteration process in Eqs. \eqref{eq:Iterproa}--\eqref{eq:Iterprod}, we have 
    \begin{align*}
        w_t-w_{xx}+w&\geq 0 \quad \text{in} \quad \Omega_T,\\
        w_x(0, t)&=\frac{-\alpha}{1+\bigl[\underline{q}^{(0)}(0,t)\bigr]^m}-\underline{q}^{(0)}_x (0, t)\geq 0,\\
        w(x,0)&\geq 0 \quad x\in 
        \Omega.
    \end{align*}
    Since $|\underline{q}^{(1)}(x,t)|\leq M_1\exp{(bx^2)}$ and $|\underline{q}^{(0)}(x,t)|\leq C\exp{(bx^2)}$ as $x\rightarrow{\infty}$ for some constants $M_1, C$ and $b$, we have for some $\delta>b$
    \begin{equation*}
\limsup_{R\rightarrow{\infty}}\bigl[\exp{(-\delta R^2)}\left(\min\bigl\{w(x,t): x=R\bigr\}\right)\bigr]\geq \lim_{R\rightarrow{\infty}}\bigl[-(M_1+C)\exp{((b-\delta)R^2)}\bigr] =0.  
    \end{equation*}
    From the Phragman-Lindelof principle, we conclude that $\underline{q}^{(1)}\geq \underline{q}^{(0)}$ in $\Omega_T$. From Theorem \ref{th:Main}, $\underline{q}^{(k)}$ exists and satisfies the growth condition.
    By method of induction,  the induction hypothesis is 
    \begin{equation*}
        \check{q}=\underline{q}^{(0)}\leq \underline{q}^{(1)}\leq ...\leq \underline{q}^{(k-1)}\leq \underline{q}^{(k)}.
    \end{equation*}
    Let $w=\underline{q}^{(k+1)}-\underline{q}^{(k)}$. Then 
    \begin{align*}
       w_t-w_{xx}+w&= 0 \quad \text{in} \quad \Omega_T,\\
        w_x(0, t)&=\frac{-\alpha}{1+\bigl[\underline{q}^{(k)}(0,t)\bigr]^m}+\frac{\alpha}{1+\bigl[\underline{q}^{(k-1)}(0,t)\bigr]^m}\geq 0,\\
        w(x,0)&=0 \quad x\in 
        \Omega. 
    \end{align*}
    Since $\underline{q}^{(k)}$ and $\underline{q}^{(k+1)}$ satisfy the growth condition, there exists $C_1$ and $\delta >b$ such that 
    \begin{equation*}
\limsup_{R\rightarrow{\infty}}\bigl[\exp{(-\delta R^2)}\left(\min\bigl\{w(x,t): x=R\bigr\}\right)\bigr]\geq \lim_{R\rightarrow{\infty}}\bigl[-C_1\exp{((b-\delta)R^2)}\bigr] =0.  
    \end{equation*}
    Phragman-Lindelof principle concludes that $\underline{q}^{(k)} \leq \underline{q}^{(k+1)}$ in $\Omega_T$. Hence, we proved $\{\underline{q}^{(k)}\}$ is nondecreasing.
    
    Since, 
    \begin{equation*}
        \underline{q}_x^{(k)}(0,t)=\frac{-\alpha}{1+\bigl[\underline{q}^{(k-1)}(0,t)\bigr]^m}\leq \frac{-\alpha}{1+\bigl[\underline{q}^{(k)}(0,t)\bigr]^m},
    \end{equation*}
      $\underline{q}^{(k)}$ is the lower solution to Eqs. \eqref{eq:ModNodela}--\eqref{eq:ModNodeld} for each $k$.\\
    \textbf{Step 2.} Here, we prove that the sequence $\{\overline{q}^{(k)}\}$ is nonincreasing and serves as an upper solution. Using similar reasoning as in Step 1, we obtain 
    \begin{equation*}
        \overline{q}^{(k+1)}\leq \overline{q}^{(k)}\leq ...\leq \overline{q}^{(1)}\leq \overline{q}^{(0)}=\hat{q}.
    \end{equation*}
    And, $\overline{q}^{(k)}$ is an upper solution to Eqs. \eqref{eq:ModNodela}--\eqref{eq:ModNodeld} for each $k$.\\
    \textbf{Step 3.} Here, we prove that $\underline{q}^{(k)}\leq \overline{q}^{(k)}$ for each $k$. Let $w=\overline{q}^{(k)}-\underline{q}^{(k)}$. Then 
    \begin{align*}
        w_t-w_{xx}+w&\geq0 \quad \text{in} \quad \Omega_T,\\
        w_x(0, t)&= \frac{-\alpha}{1+\bigl[\overline{q}^{(k-1)}(0,t)\bigr]^m}+\frac{\alpha}{1+\bigl[\underline{q}^{(k-1)}(0,t)\bigr]^m},\\
        w(x,0)&\geq0 \quad x\in\Omega.
    \end{align*}
    Since, $\underline{q}^{(k-1)}\leq \underline{q}^{(k)}$ and $\overline{q}^{(k)}\leq \overline{q}^{(k+1)}$, we have 
    \begin{equation*}
        w_x(0, t)= \frac{-\alpha}{1+\bigl[\overline{q}^{(k-1)}(0,t)\bigr]^m}+\frac{\alpha}{1+\bigl[\underline{q}^{(k-1)}(0,t)\bigr]^m}\geq \frac{-\alpha}{1+\bigl[\overline{q}^{(k)}(0,t)\bigr]^m}+\frac{\alpha}{1+\bigl[\underline{q}^{(k)}(0,t)\bigr]^m}.
    \end{equation*}
   The function $g(0,t,q):[0,T]\times \langle\check{q}, \hat{q}\rangle \rightarrow \mathbb{R}$ defined by $g(0,t,q)=\frac{-\alpha}{1+[q(0,t)]^m}$ is Lipschitz continuous with a Lipschitz constant $\Gamma$. Hence, we have 
    \begin{align*}
        w_x(0, t)+\Gamma w(0,t)&\geq \frac{-\alpha}{1+\bigl[\underline{q}^{(k)}\bigr]^m}+\frac{\alpha}{1+\bigl[\overline{q}^{(k)}\bigr]^m}+ \Gamma w(0,t)\geq 0, 
    \end{align*}
    With the Phragman-Lindelof principle, we conclude that $\underline{q}^{(k)}\leq \overline{q}^{(k)}$ for each $k$.\\
    \textbf{Step 4.} Since the sequence $\{\underline{q}^{(k)}\}$ and $\{\overline{q}^{(k)}\}$ are monotonic and bounded, therefore there exist pointwise limits $\underline{q}$ and $\overline{q}$, such that
    \begin{equation*}
        \lim_{k\rightarrow \infty}\overline{q}^{(k)}=\overline{q} \quad \text{and} \quad  \lim_{k\rightarrow \infty}\underline{q}^{(k)}=\underline{q}
    \end{equation*}
\end{proof}

\section{Green's function estimation}\label{ap:ap4}
For fixed $t>s>0$, $\int_0^t G(x, t ; 0,s)  \mathrm{d} s$  is finite. \\
    We have, 
   \nonumber \begin{align}
 \left|\int_0^t G(x, t ; 0, s) \mathrm{d}s\right| &=
 2 \left|\int_0^t \frac{e^{-t}}{\sqrt{4 \pi(t-s)}} \text{exp}\left(\frac{-x^2}{4(t-s)}\right) \mathrm{d} s\right|.
    \end{align}

We use the following inequality, for every $n>0, \quad \exists \  C_n$ such that,

\begin{equation*}
a^n e^{-b a^2} \leq C_n e^{-\frac{b}{2} a^2} \quad \forall a>0. 
\end{equation*}
It implies that,
\begin{align*}
   \frac{1}{\sqrt{4(t-s)}} \exp \left(\frac{-x^2}{4(t-s)}\right) \leq C_1 \text{exp}\left(-\frac{x^2}{8(t-s)}\right) \leq C_1.
\end{align*}
Hence, 
\begin{equation*}
   \Bigg| \int_0^t G(x, t ; 0,s)  \mathrm{d} s \Bigg| \leq 2 C_1t\exp{(-t)}.
\end{equation*}

\section{Calculation for Proposition~\ref{pr:lusol1}}\label{ap:ap5}

Let \begin{equation*}
    \check{q}(x)=c\exp{(-\zeta x-\beta x^{\gamma})}
\end{equation*}
Then, \begin{align}
   \check{q}_t&=0,\\
   \check{q}_x&=\check{q}(-\zeta-\beta \gamma x^{\gamma-1})\\
     \check{q}_{xx}&=-\beta \gamma (\gamma-1)x^{\gamma-2}\check{q}+(\zeta+\beta \gamma x^{\gamma-1})^2\check{q}.
\end{align}
By substituting above equations in Eq. \eqref{eq:ModNodela}, we get 
\begin{equation*}
    \check{q}_t-\check{q}_{xx}+\check{q}=\check{q}\left( \beta \gamma (\gamma-1)x^{\gamma-2}-\zeta^2-\beta^2 \gamma^2x^{2\gamma-2}-2\zeta\beta \gamma x^{\gamma-1} +1\right). 
\end{equation*}
Since, $\gamma \geq 2, \beta >1$ and $\zeta \geq \sqrt{\beta \gamma(\gamma-1)+1}$, we have 
\begin{equation*}
    \check{q}_t-\check{q}_{xx}+\check{q}\leq 0.
\end{equation*}
Since $c+c^{m+1}=\alpha$ and $\zeta\geq 1$, 
\begin{align}
    \zeta c+\zeta c^{m+1}&\geq \alpha, \\
    -\zeta c-\zeta c^{m+1}&\leq -\alpha,\\ 
    -\zeta c&\leq -\frac{\alpha}{1+(c)^m}. 
\end{align}
This implies that, $\check{q}_x(0)\leq -\frac{\alpha}{1+[\check{q}(0)]^m}$. From the above calculation and Hypothesis~\ref{hy:ini} we proved that $\check{q}(x)=c\exp{(-\zeta x-\beta x^{\gamma})}$ is a lower solution to  Eqs. \eqref{eq:ModNodela}--\eqref{eq:ModNodeld}. It is straightforward to see that $\hat{q}=M$ is an upper solution to Eqs. \eqref{eq:ModNodela}--\eqref{eq:ModNodeld}.

\section{3D plots for visualizing numerical examples}\label{ap:ap3} 
We present three-dimensional visualizations for Examples 1--4, depicting simulations over a truncated domain to highlight oscillatory behavior. Figs. \ref{fig:eg1}--\ref{fig:eg4} present numerical simulations over a large time scale, clearly illustrating the sustained or decreasing amplitudes in the solutions as time progresses.
\begin{figure}[h]
 \centering  \includegraphics[width=\textwidth]{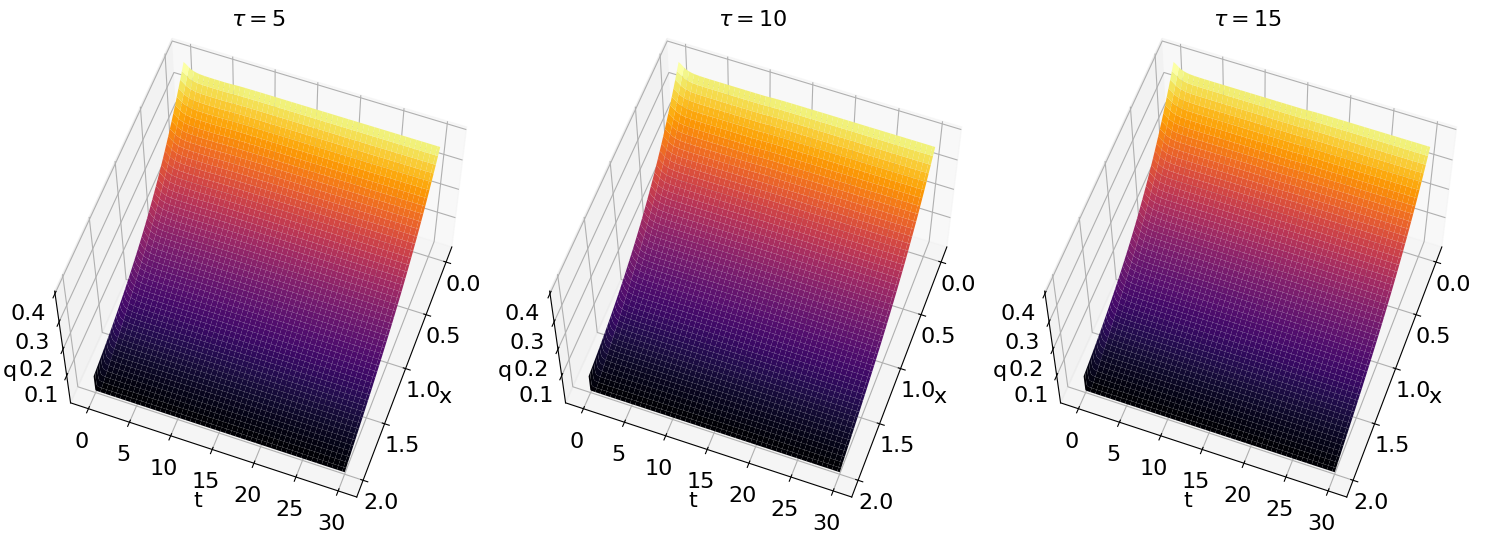}
  \caption{Example 1 (Protein concentration $q(x,t)$  for $m=4$ and $\alpha=0.4$ )} 
  \label{fig:3deg1}  
\end{figure}  
\begin{figure}[h]
 \centering
\includegraphics[width=\textwidth]{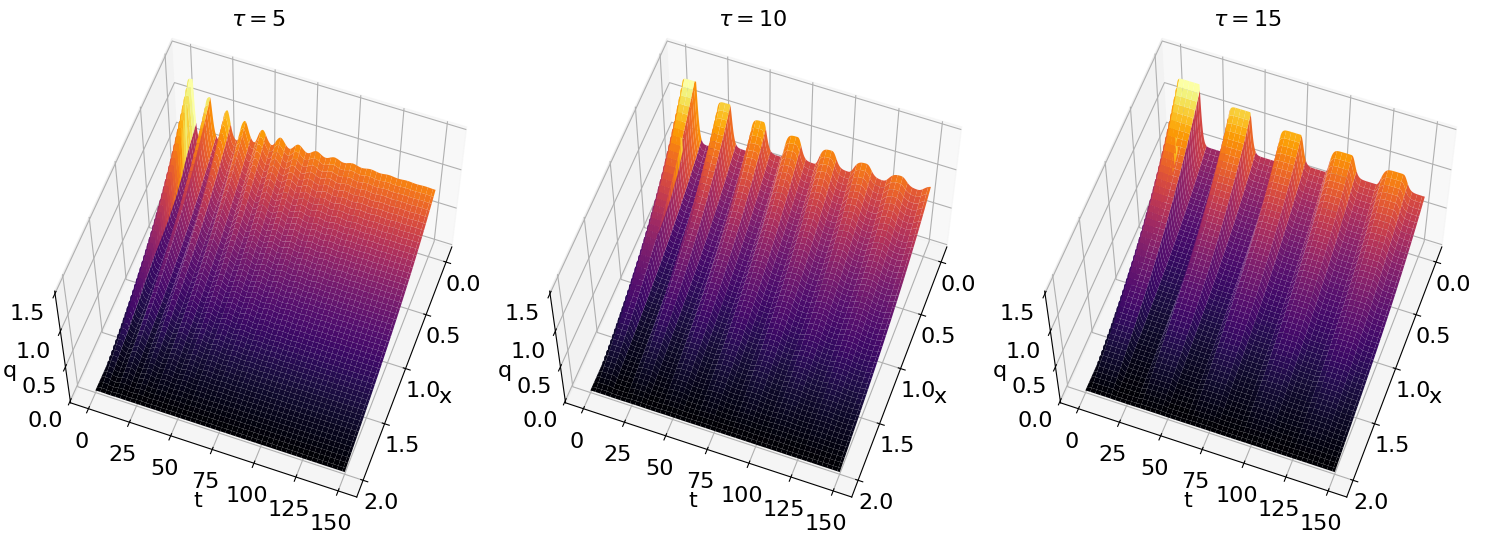}
  \caption{Example 2 (Protein concentration $q(x,t)$  for $m=2$ and $\alpha=1.6$ )} 
  \label{fig:3deg2}  
\end{figure} 
\begin{figure}[h]
 \centering
\includegraphics[width=\textwidth]{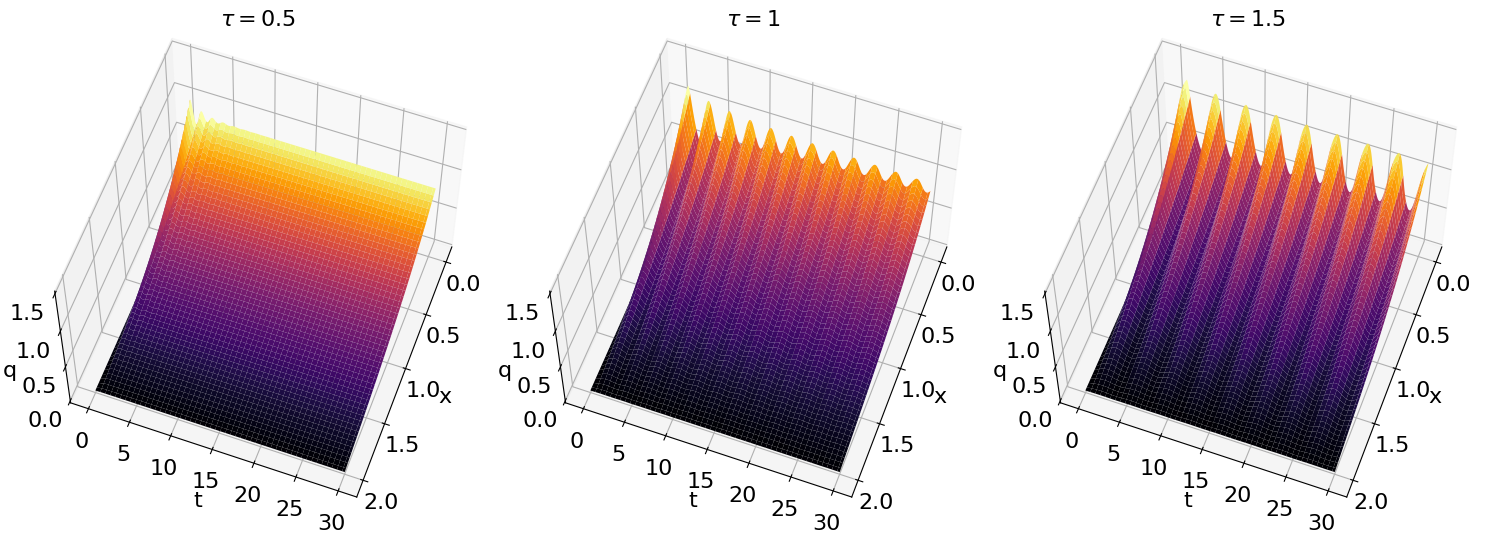}
  \caption{Example 3 (Protein concentration $q(x,t)$  for $m=4$ and $\alpha=1.5$ )} 
  \label{fig:3deg3}  
\end{figure} 
\vspace{2cm}
\begin{figure}[h]
 \centering
\includegraphics[width=\textwidth]{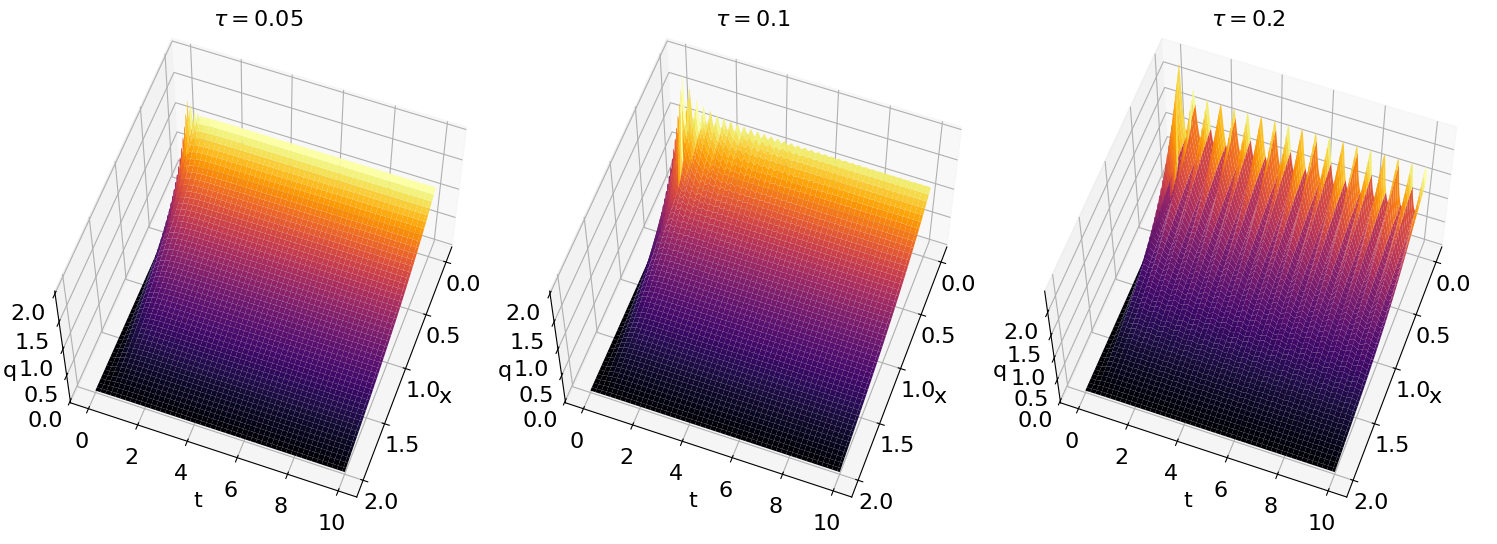}
  \caption{Example 4 (Protein concentration $q(x,t)$  for $m=6$ and $\alpha=5$ )} 
  \label{fig:3deg4} 
\end{figure}


\newpage
\begin{thebibliography}{9}

\bibitem[Andreev~\& Racheva, 2007]{andrev2007}
Andreev, A.~B. \& Racheva, M.~R. [2007]
``On the numerical solutions of eigenvalue problems in unbounded domains,''
in {\it Numerical Methods and Applications}, pp. 500--507, \url{https://doi.org/10.1007/978-3-540-70942-8_60}.

\bibitem[Arrieta~\& Carvalho~\& Rodriguez-Bernal, 1999]{arrieta1999}
Arrieta, J.~M., Carvalho, A.~N. \& Rodriguez-Bernal, A. [1999]
``Parabolic problems with nonlinear boundary conditions and critical nonlinearities,''
{\it J. Differ. Equations} {\bf 156}(2), 376--406,
\url{https://doi.org/10.1006/jdeq.1998.3612}.

\bibitem[Bertozzi, 2001]{bertozzi2001}
Bertozzi, A.~L. [2001]
{\it Vorticity and Incompressible Flow. Cambridge Texts in Applied Mathematics} (Cambridge University Press, UK),
\url{https://doi.org/10.1017/CBO9780511613203}.

\bibitem[Brezis, 2011]{brezis2011}
Brezis, H. [2011]
{\it Functional Analysis, Sobolev Spaces and Partial Differential Equations} (Springer Science \& Business Media, USA), \url{https://doi.org/10.1007/978-0-387-70914-7}.

\bibitem[Chen~\& Lou~\& Wei, 2018]{chen2018}
Chen, S., Lou, Y. \& Wei, J. [2018]
``Hopf bifurcation in a delayed reaction--diffusion--advection population model,''
{\it J. Differ. Equations} {\bf 264}(8), 5333--5359,
\url{https://doi.org/10.1016/j.jde.2018.01.008}.

\bibitem[Chen~\& Shi, 2012]{chen2012}
Chen, S. \& Shi, J. [2012]
``Stability and Hopf bifurcation in a diffusive logistic population model with nonlocal delay effect,''
{\it J. Differ. Equations} {\bf 253}(12), 3440--3470,
\url{https://doi.org/10.1016/j.jde.2012.08.031}.

\bibitem[Chung, 2002]{chung2002}
Chung, T.~J. [2002]
{\it Computational Fluid Dynamics} (Cambridge University Press, UK),
\url{https://doi.org/10.1017/CBO9780511606205}.

\bibitem[Evans(2022)]{evans2022}
Evans, L.~C. [2022]
{\it Partial Differential Equations}, 2nd Ed. (American Mathematical Society, USA), \url{https://doi.org/10.1090/gsm/019}.

\bibitem[Friedman(2008)]{friedman2008}
Friedman, A. [2008]
{\it Partial Differential Equations of Parabolic Type} (Courier Dover Publications, USA).

\bibitem[Guo, 2015]{guo2015}
Guo, S. [2015]
``Stability and bifurcation in a reaction--diffusion model with nonlocal delay effect,''
{\it J. Differ. Equations} {\bf 259}(4), 1409--1448,
\url{https://doi.org/10.1016/j.jde.2015.03.006}.

\bibitem[Guo, 2021]{guo2021}
Guo, S. [2021]
``Bifurcation in a reaction-diffusion model with nonlocal delay effect and nonlinear boundary condition,''
{\it J. Differ. Equations} {\bf 289}, 236--278,
\url{https://doi.org/10.1016/j.jde.2021.04.021}.

\bibitem[Guo, 2023]{guo2023}
Guo, S. [2023]
``Global dynamics of a Lotka-Volterra competition-diffusion system with nonlinear boundary conditions,''
{\it J. Differ. Equations} {\bf 352}, 308--353,
\url{https://doi.org/10.1016/j.jde.2023.01.010}.

\bibitem[Guo~\& Li~\& Sounvoravong, 2021]{guo2021oscillatory}
Guo, S., Li, S. \& Sounvoravong, B. [2021]
``Oscillatory and stationary patterns in a diffusive model with delay effect,''
{\it Int. J. Bifurcation Chaos} {\bf 31}(3), 2150035,
\url{https://doi.org/10.1142/S0218127421500358}.

\bibitem[Guidotti~\& Merino, 1997]{guidotti1997}
Guidotti, P. \& Merino, S. [1997]
``Hopf bifurcation in a scalar reaction diffusion equation,''
{\it J. Differ. Equations} {\bf 140}(1), 209--222,
\url{https://doi.org/10.1006/jdeq.1997.3307}.

\bibitem[Holmes~\& Lewis~\& Banks~\& Veit, 1994]{holmes1994}
Holmes, E.~E., Lewis, M.~A., Banks, J.~E. \& Veit, R.~R. [1994]
``Partial differential equations in ecology: spatial interactions and population dynamics,''
{\it Ecology} {\bf 75}(1), 17--29,
\url{https://doi.org/10.2307/1939378}.

\bibitem[Hui~\& Liu~\& Zhao, 2022]{hui2022}
Hui, Y., Liu, Y. \& Zhao, Z. [2022]
``Hopf bifurcation in a delayed equation with diffusion driven by carrying capacity,''
{\it Mathematics} {\bf 10}(14), 2382,
\url{https://doi.org/10.3390/math10142382}.

\bibitem[Li~\& Guo, 2024]{li2024}
Li, C. \& Guo, S. [2024]
``Bifurcation and stability of a reaction--diffusion--advection model with nonlocal delay effect and nonlinear boundary condition,''
{\it Nonlinear Anal. Real World Appl.} {\bf 78}, 104089,
\url{https://doi.org/10.1016/j.nonrwa.2024.104089}.

\bibitem[Magal~\& Ruan, 2018]{magal2018}
Magal, P. \& Ruan, S. [2018]
{\it Theory and Applications of Abstract Semilinear Cauchy Problems} (Springer, USA), \url{https://doi.org/10.1007/978-3-030-01506-0}.

\bibitem[Mahaffy~\& Pao, 1984]{mahaffy1984}
Mahaffy, J.~M. \& Pao, C.~V. [1984]
``Models of genetic control by repression with time delays and spatial effects,''
{\it J. Math. Biol.} {\bf 20}(1), 39--57,
\url{https://doi.org/10.1007/BF00275860}.

\bibitem[Memory, 1989]{memory1989}
Memory, M.~C. [1989]
``Bifurcation and asymptotic behavior of solutions of a delay-differential equation with diffusion,''
{\it SIAM Journal on Mathematical Analysis} {\bf 20}(3), 533--546,
\url{https://doi.org/10.1137/0520037}.

\bibitem[Miller et~al., 2024]{miller2024}
Miller, E.~M., Chan, T.~C.~D., Montes-Matamoros, C., Sharif, O., Pujo-Menjouet, L. \& Lindstrom, M.~R. [2024]
``Oscillations in neuronal activity: a neuron-centered spatiotemporal model of the Unfolded Protein Response in prion diseases,''
{\it Bulletin of Mathematical Biology} {\bf 86}(7), 82,
\url{https://doi.org/10.1007/s11538-024-01307-y}.

\bibitem[Pao, 1987]{pao1987}
Pao, C.~V. [1987]
``On a coupled reaction diffusion system with time delays,''
{\it SIAM J. Math. Anal.} {\bf 18}(4), 1026--1039,
\url{https://doi.org/10.1137/0518077}.

\bibitem[Pao, 2002]{pao2002}
Pao, C.~V. [2002]
``Time delayed parabolic systems with coupled nonlinear boundary conditions,''
{\it Proc. Amer. Math. Soc.} {\bf 130}(4), 1079--1086, \url{https://doi.org/10.1090/S0002-9939-01-06319-5}.

\bibitem[Pao(2012)]{pao2012}
Pao, C.~V. [2012]
{\it Nonlinear Parabolic and Elliptic Equations} (Springer Science \& Business Media, USA), \url{ https://doi.org/10.1007/978-1-4615-3034-3}.

\bibitem[Pao~\& Ruan, 2007]{pao2007}
Pao, C.~V. \& Ruan, W.~H. [2007]
``Positive solutions of quasilinear parabolic systems with nonlinear boundary conditions,''
{\it J. Math. Anal. Appl.} {\bf 333}(1), 472--499,
\url{https://doi.org/10.1016/j.jmaa.2006.10.005}.

\bibitem[Pennes, 1948]{Pennes1948}
Pennes, H.~H. [1948]
``Analysis of tissue and arterial blood temperatures in the resting human forearm,''
{\it J. Appl. Physiol.} {\bf 1}(2), 93--122,
\url{https://doi.org/10.1152/jappl.1948.1.2.93}.

\bibitem[Tiwari~\& Klar~\& Russo, 2020]{tiwari2020}
Tiwari, S., Klar, A. \& Russo, G. [2020]
``Interaction of rigid body motion and rarefied gas dynamics based on the BGK model,''
{\it Mathematics in Engineering} {\bf 2}(2), 203--229,
\url{https://doi.org/10.3934/mine.2020010}.

\bibitem[Travis~\& Webb, 1974]{travis1974}
Travis, C.~C. \& Webb, G.~F. [1974]
``Existence and stability for partial functional differential equations,''
{\it Trans. Amer. Math. Soc.} {\bf 200}, 395--418,
\url{https://doi.org/10.2307/1997265}.

\bibitem[Wu(2012)]{wu2012}
Wu, J. [2012]
{\it Theory and Applications of Partial Functional Differential Equations} (Springer Science \& Business Media, USA), \url{https://doi.org/10.1007/978-1-4612-4050-1}.
\end{thebibliography}
\end{document}